\documentclass[11pt]{amsart}

\usepackage[a4paper,hmargin=3cm,vmargin=3cm]{geometry}
\usepackage{amsfonts,amssymb,amscd,amstext}
\usepackage{graphicx}
\usepackage[dvips]{epsfig}
\usepackage[latin1]{inputenc}

\usepackage{fancyhdr}
\usepackage{times}

\usepackage{enumerate}
\usepackage{titlesec}
\usepackage{mathrsfs}
\usepackage{stmaryrd}

\def\R{\mathbb{R}}

\def\N{\mathbb{N}}

\newcommand{\ben}{\begin{enumerate}}
\newcommand{\bit}{\begin{itemize}}
\newcommand{\een}{\end{enumerate}}
\newcommand{\eit}{\end{itemize}}

\newcommand{\ed}{\end{document}}

\newcommand{\cal}{\mathcal}

\def\cU{\mathcal{U}}

\def\cS{\mathcal{S}}
\def\cC{\mathcal{C}}

\def\cW{\mathcal{W}}
\def\cS{\mathcal{S}}

\let\8=\infty \let\0=\emptyset  
\let\hat=\widehat

\let\tilde=\widetilde

\let\landa=\lambda

\let\parc=\partial

\def\ep{\varepsilon}

\def\landa{\lambda}

\def\flecha{\rightarrow}
\def\esiz{\langle}
\def\esde{\rangle}
\def\s{\hbox{\bb S}}
\newcommand{\fl}{\longrightarrow}
\def\r{\mathbb{R}}
\def\S{\Sigma}

\def\cte.{\mathop{\rm cte.}\nolimits}

\def\N{\mathbb{N}}

\def\R{\mathbb{R}}

\def\S{\mathbb{S}}

\newfont{\bb}{msbm10 at 12pt}

\def\s{\hbox{\bb S}}

\def\n{\mathbb{N}}
\def\r{\mathbb{R}}

\titleformat{\section}
{\filcenter\bfseries\large} {\thesection{.}}{0.2cm}{}
\titleformat{\subsection}[runin]
{\bfseries} {\thesubsection{.}}{0.15cm}{}[.]
\titleformat{\subsubsection}[runin]
{\em}{\thesubsubsection{.}}{0.15cm}{}[.]

\usepackage[up,bf]{caption}

\newtheorem{theorem}{Theorem}[section]
\newtheorem{lemma}[theorem]{Lemma}
\newtheorem{proposition}[theorem]{Proposition}

\newtheorem{remark}[theorem]{Remark}
\newtheorem{corollary}[theorem]{Corollary}

\newtheorem{assertion}[theorem]{Assertion}

\theoremstyle{definition}

\numberwithin{equation}{section}
\numberwithin{figure}{section}

\usepackage{color}

\begin{document}
\fancyhead[LO]{Constant anisotropic mean curvature surfaces}
\fancyhead[RE]{José A. Gálvez, Pablo Mira, Marcos P. Tassi}
\fancyhead[RO,LE]{\thepage}

\thispagestyle{empty}

\begin{center}
{\bf \LARGE Complete surfaces of constant anisotropic\\[0.2cm] mean curvature}
\vspace*{5mm}

\hspace{0.2cm} {\Large José A. Gálvez, Pablo Mira, Marcos P. Tassi}
\end{center}

\footnote[0]{
\noindent \emph{Mathematics Subject Classification}: 53A10, 53C42. \\ \mbox{} \hspace{0.25cm} \emph{Keywords}: Constant anisotropic mean curvature, Wulff shape, classification theorems, multigraph.}



\vspace*{7mm}

\begin{quote}
{\small
\noindent {\bf Abstract}\hspace*{0.1cm}
We study the geometry of complete immersed surfaces in $\R^3$ with constant anisotropic mean curvature (CAMC). Assuming that the anisotropic functional is uniformly elliptic, we prove that: (1) planes and CAMC cylinders are the only complete surfaces with CAMC whose Gauss map image is contained in a closed hemisphere of $\S^2$; (2) Any complete surface with non-zero CAMC and whose Gaussian curvature does not change sign is either a CAMC cylinder or the Wulff shape, up to a homothety of $\R^3$; and (3) if the Wulff shape $\cW$ of the anisotropic functional is invariant with respect to three linearly independent reflections in $\R^3$, then any properly embedded surface of non-zero CAMC, finite topology and at most one end is homothetic to $\cW$.



\vspace*{0.1cm}

}
\end{quote}


\section{Introduction}
Let $F:\S^2\flecha\R$ be a smooth positive function on the unit $2$-sphere. Then, $F$ defines the following functional on the space of immersed oriented surfaces in $\R^3$:
\begin{equation}\label{11}
{\mathcal F}(\Sigma)=\int_{\Sigma}F(N)\,d\Sigma,
\end{equation}
where $N:\Sigma\fl\s^2$ is the unit normal of $\Sigma$ and $d\Sigma$ denotes its area element. When $F=1$, \eqref{11} is the area functional. The Euler-Lagrange equation associated to \eqref{11} is uniformly elliptic when $F$ satisfies the convexity condition 
\begin{equation}\label{concon}
\nabla^2 F + F\, \esiz,\esde_{\S^2} >0,
\end{equation}
where $\nabla^2 F$ is the intrinsic Hessian of $F$ in $\S^2$, $\esiz,\esde_{\S^2}$ is the Riemannian metric of $\S^2$, and $>0$ means that the symmetric bilinear form given in \eqref{concon} is positive definite.

\emph{The ellipticity condition \eqref{concon} will be assumed from now on}. It is equivalent to the fact that the map 
$\eta:\s^2\fl\r^3$ given by 
\begin{equation}\label{eta}
\eta(p)=\nabla F(p)+F(p)p
\end{equation}
is a diffeomorphism onto a smooth, compact, strictly convex sphere $\cW\subset \R^3$; here $\nabla F$ denotes the gradient of $F$ in $\S^2$. The ovaloid $\cW$ is called the \emph{Wulff shape} associated to $F$. The exterior unit normal of $\cW$ is given by $\eta^{-1}:\cW\flecha \S^2$. If $F=1$, then $\cW$ is the unit sphere of $\R^3$. 

The critical points of \eqref{11}, with or without a volume constraint, have been deeply studied; they admit a geometric characterization that we explain next. 

For any immersed oriented surface $\Sigma$ in $\R^3$ with Gauss map $N$, we can define the \emph{anisotropic Gauss map} of $\Sigma$ as the map 
\begin{equation}\label{def:agm}
\nu:\Sigma\fl{\mathcal W},\quad \nu=\eta\circ N,
\end{equation}
that sends each $p\in\Sigma$ to the unique point $\nu(p)\in{\mathcal W}$ with its same oriented tangent plane. 
Then, given $p\in \Sigma$, the \emph{anisotropic mean curvature} $H$ of $\Sigma$ at $p$ is the trace of the endomorphism $A_p:=-d\nu_p$. When $F=1$, the anisotropic mean curvature of $\Sigma$ is twice its usual mean curvature (since we are using the trace, and not one half of it). 

The Wulff shape $\cW$ has constant anisotropic mean curvature equal to $-2$ with respect to its exterior unit normal. The anisotropic mean curvature of $\cW$ with respect to its interior unit normal is not, in general, constant. Planes have vanishing anisotropic mean curvature, for any orientation. 

With these definitions in mind, the following geometric equivalence holds: the anisotropic minimal surfaces (i.e. the immersed surfaces in $\R^3$ with vanishing anisotropic mean curvature) are exactly the critical points of the functional \eqref{11}. Analogously, the surfaces with constant anisotropic mean curvature (CAMC) $H_0\neq 0$ are the critical points of \eqref{11} under a fixed volume constraint; equivalently, they are the critical points of 
\begin{equation}\label{2}
{\mathcal F}_0(\Sigma)=\int_{\Sigma}\left(F(N)+\frac{1}{3}\,H_0\langle\psi,N\rangle\right)\,d\Sigma,
\end{equation}
where $\psi$ denotes a parametrization of $\Sigma$ with Gauss map $N$.

The class of CAMC surfaces has been widely studied, specially from the viewpoint of measure theory and convex analysis. The case of anisotropic minimal surfaces has also received some classical contributions from a more geometric viewpoint, see e.g. \cite{Finn,J,JS1,JS2}. The geometry of surfaces with non-zero CAMC has been recently studied in more detail in many works; see, e.g., \cite{Clarenz1,Clarenz2,GM,CCC,He2,He5,He4,KP,KP3,Koiso3,KoisoPalmer,Kuhns,Lira,Palmer} and references therein.

Some of these previous works have provided a good understanding of the basic geometry of compact (without boundary) CAMC surfaces. For instance, there exist anisotropic extensions of the classical theorems of CMC surface theory by Barbosa-do Carmo \cite{BC}, Alexandrov \cite{A} and Hopf \cite{Ho0,Ho}, that classify, respectively, the compact surfaces with CAMC that are stable (Palmer, \cite{Palmer}), embedded (He-Li-Ma-Ge, \cite{CCC}) or have genus zero (Koiso-Palmer, \cite{KoisoPalmer}; see also Gálvez-Mira \cite{GM3}).

In contrast, the global geometry of \emph{complete} immersed CAMC surfaces is quite less understood, and many classical theorems of CMC surface theory still do not have an anisotropic analogue. Our objective in this paper is to give an extension to the anisotropic setting of three of these classical theorems, namely (see \cite{Hoffman2,Klotz,Me}):

{\bf Theorem A (Hoffman-Osserman-Schoen)} \emph{Planes and cylinders are the only complete surfaces with constant mean curvature in $\R^3$ whose Gauss map image lies in a closed hemisphere of $\S^2$}.

\vspace{0.2cm}

{\bf Theorem B (Klotz-Osserman).}  \emph{Spheres and cylinders are the only complete surfaces with non-zero constant mean curvature in $\R^3$ whose Gaussian curvature does not change sign}.

\vspace{0.2cm}

{\bf Theorem C (Meeks).}  \emph{Spheres are the only properly embedded surfaces in $\R^3$ with non-zero constant mean curvature, finite topology and at most one end.}

\vspace{0.2cm}

In this paper we will prove Theorems A and B for \emph{any} $F$ (subject to \eqref{concon}), and Theorem C for choices of $F$ that are symmetric with respect to three linearly independent directions. We next explain our results in more detail, and give an outline of the paper. Recall that, in all that follows, $F$ is a positive smooth function on $\S^2$ that satisfies the ellipticity condition \eqref{concon}.

In Section \ref{sec:2} we study complete surfaces with CAMC and bounded second fundamental form. First, we will prove a compactness theorem for this type of surfaces (Theorem \ref{compacidad}), based on elliptic theory. Then, we will give a general \emph{curvature estimate}, by proving that there exists a uniform a priori estimate on the norm of the second fundamental form of any complete surface with non-zero CAMC whose Gauss map image omits an open set of $\S^2$; see Theorem \ref{curvatura}. For that, we will use Theorem \ref{compacidad} and a rescaling argument. The proofs are inspired in previous work by the first two authors with A. Bueno on complete surfaces in $\R^3$ of prescribed mean curvature (not necessarily constant); see \cite{BGM}. 

In Section \ref{sec:3}, which contains the core result of the paper, we will study complete \emph{multigraphs} of CAMC. Here, following the standard terminology, we will say that an immersed surface $\Sigma$ in $\R^3$ is a \emph{multigraph} if there exists a plane $P\subset \R^3$ such that $\Sigma$ is locally a graph over $P$ around each point of $\Sigma$. Equivalently, the Gauss map image of $\Sigma$ lies in an open hemisphere of $\S^2$. Obviously, every graph is a multigraph. Our main result in Section \ref{sec:3} is Theorem \ref{multigrafo}, namely:

\begin{center}
\emph{Any complete multigraph with constant anisotropic mean curvature is a plane.}
\end{center}

This result can be seen as a kind of general Bernstein-type theorem for CAMC surfaces. For anisotropic minimal surfaces, this is a theorem by Jenkins \cite{J}. It was proved by Koiso-Palmer in \cite{KP} in the particular case when \eqref{11} is close to the area functional in a suitable sense, for complete \emph{stable} CAMC surfaces (not necessarily multigraphs).

The proof of Theorem \ref{multigrafo} differs completely from the approaches by Jenkins, Koiso-Palmer and Hoffman-Osserman-Schoen. Instead, it relies on ideas developed by Hauswirth, Rosenberg and Spruck \cite{Hauswirth}, and by Espinar and Rosenberg \cite{ER} (see also \cite{Manzano}), in the study of complete surfaces of constant mean curvature in homogeneous three-manifolds (see \cite{DHM,FM} for more information on this theory). 

For the proof of Theorem \ref{multigrafo}, we will use the following construction (see also \cite{GM}). Given a unit vector $v_0\in\s^2$, let ${\cal W}_0$ denote the set of points of the Wulff shape ${\cal W}$ whose unit normal is orthogonal to $v_0$. Then, the flat cylinder
$$
{\cal C}_{v_0}=\{p+\lambda v_0\in\r^3:\ p\in{\cal W}_0,\lambda\in\r\}
$$
is smooth and has CAMC equal to $-1$, with respect to its exterior unit normal; it will be called a \emph{CAMC cylinder}. 

As a consequence of Theorem \ref{multigrafo} we will extend Theorems A and B above to the general anisotropic case, and prove:

\begin{enumerate}
\item
\emph{Any complete immersed CAMC surface whose Gauss map image is contained in a closed hemisphere of $\S^2$ is a plane or a CAMC cylinder.} (Corollary \ref{corolario1}).\vspace{0.2cm}
\item
\emph{Any complete immersed surface with non-zero CAMC, and whose Gaussian curvature does not change sign, is a CAMC cylinder or the Wulff shape, up to homotheties.} (Theorem \ref{kno}).
\end{enumerate}

In Section \ref{sec:4} we will study properly embedded surfaces of non-zero CAMC. We will start by proving some height estimates for graphs of non-zero CAMC and planar boundary (Lemma \ref{compactas}, Lemma \ref{meeks} and Theorem \ref{meeks2}). Then, we will derive geometric consequences of these estimates and of Meeks' \emph{separation lemma}, in the case that the Wulff shape $\cW$ is invariant with respect to some Euclidean reflection. Our final result, Theorem \ref{final}, will characterize the Wulff shape (up to homothety) as the only properly embedded surface in $\R^3$ with non-zero CAMC, finite topology and at most one end, assuming that $\cW$ is symmetric with respect to three linearly independent planes of $\R^3$. This gives a wide extension of Meeks' result (Theorem C in the introduction) to the anisotropic case. For a similar result in the case of properly embedded surfaces in $\R^3$ with prescribed (non-constant) mean curvature, see \cite{BGM}.

\section{CAMC surfaces with bounded second fundamental form}\label{sec:2}

In all that follows, we let $F$ be a smooth positive function on $\S^2$ that satisfies the uniform ellipticity condition \eqref{concon}, and we let $\cW\subset \R^3$ denote its associated Wulff shape, which is an ovaloid in $\R^3$. Whenever we write CAMC, it is understood that we mean CAMC with respect to the function $F$. Unless otherwise stated, by a \emph{surface} in $\R^3$ we mean an immersed one, i.e., not necessarily embedded.

Let $\Sigma$ be an oriented surface in $\R^3$, and let $\nu:\Sigma\flecha \cW$ denote its anisotropic Gauss map, defined in \eqref{def:agm}. For each $p\in\Sigma$, the endomorphism $A_p:=-d\nu_p$ can be described in terms of the (Euclidean) Weingarten endomorphism $S=-dN$ of $\Sigma$, and of the differential $\cS:= d\eta^{-1}$ of the outer unit normal $\eta^{-1}:\cW\flecha \S^2$ of $\cW$, as 
\begin{equation}\label{1}
A_p=({\mathcal S}_{\nu(p)})^{-1}\circ S_p. 
\end{equation}
Note that $A$ agrees with the Weingarten endomorphism of $\Sigma$ when $F\equiv 1$, since in that case ${\mathcal W}=\s^2\subset \R^3$ and ${\mathcal S}={\rm Id}$. Also, note that if for $\Sigma={\mathcal W}$ we choose the orientation given by its outward pointing unit normal $N=\eta^{-1}$, then $\nu(p)=p$ and $A_p=-{\rm Id}$ for every $p\in{\mathcal W}$.

Although $A$ is not self-adjoint, it is diagonalizable, and its eigenvalues $\landa_1,\landa_2$ are called the \emph{anisotropic principal curvatures} of $\Sigma$. We have $H=\landa_1+\landa_2$, where $H$ is the anisotropic mean curvature of $\Sigma$.

The anisotropic mean curvature behaves with respect to ambient homotheties as follows: if $\psi:\Sigma\flecha \R^3$ is an immersed surface with anisotropic mean curvature $H$ with respect to its unit normal $N$, and we consider the homothety $\Phi_c$ of $\R^3$ of ratio $c\in \R-\{0\}$, then the immersion $\tilde{\psi}:=\Phi_c\circ \psi$ has anisotropic mean curvature $\tilde{H}=\frac{1}{c} H$ with respect to the unit normal $\tilde{N}$ of $\tilde{\psi}$ given by $\tilde{N}(p)=N(p)$ for all $p\in \Sigma$. In particular, we will always be able to assume that, up to ambient homothety, a surface with non-zero CAMC has $H=-2$, i.e. the value of the constant anisotropic mean curvature of the Wulff shape for its exterior unit normal. Observe that, by the previous comments, the image of the Wulff shape with respect to the antipodal map of $\R^3$ has CAMC equal to $2$ with respect to its interior unit normal. 

Let $u(x,y)$ be a smooth function whose graph $z=u(x,y)$ has anisotropic mean curvature $H(x,y)$ with respect to its upwards-pointing unit normal, given by
$$
N=\frac{1}{\sqrt{1+p^2+q^2}}(-p,-q,1), \hspace{0.5cm} p=u_x, q=u_y.
$$
Then, we can write
$$
S(\partial_x)=-N_x=a_{11}\partial_x+a_{21}\partial_y,\qquad
S(\partial_y)=-N_y=a_{12}\partial_x+a_{22}\partial_y, 
$$
where
$$
\begin{pmatrix}
a_{11}&a_{12}\\
a_{21}&a_{22}
\end{pmatrix}
=
\frac{1}{(1+p^2+q^2)^{3/2}}
\begin{pmatrix}
1+q^2&-pq\\
-pq&1+p^2
\end{pmatrix}
\begin{pmatrix}
u_{xx}&u_{xy}\\
u_{xy}&u_{yy}
\end{pmatrix}.
$$

Besides, since $N$ only depends on $(p,q)$, we have that the inverse of the endomorphism $\cS$ defined before \eqref{1}, at a point $N(x,y)$, can be written as
$$
{\mathcal S}^{-1}_{N(x,y)}(\partial_x)=b_{11}\partial_x+b_{21}\partial_y,\qquad
{\mathcal S}^{-1}_{N(x,y)}(\partial_x)=b_{12}\partial_x+b_{22}\partial_y,
$$
where, by (\ref{eta}), the functions $b_{ij}$ only depend on $p,q$ and the derivatives up to second order of the function $F:\S^2\flecha \R$. In this way, we have from \eqref{1} that the graph $z=u(x,y)$ has CAMC  $H(x,y)=H_0$ if and only if $u$ satisfies a certain quasilinear elliptic PDE of the form
\begin{equation}\label{edp}
a(u_x,u_y)u_{xx}+b(u_x,u_y)u_{xy}+c(u_x,u_y)u_{yy}=H_0,
\end{equation}
where the coefficients $a,b,c\in C^{\8}(\R^2)$ are completely determined by $F$; here, the ellipticity of \eqref{edp} comes from \eqref{concon}, as explained in the introduction.­ This allows to use elliptic theory in order to prove the following compactness result:

\begin{theorem}\label{compacidad}
Let $\Sigma_n$ be a sequence of complete CAMC surfaces (possibly with boundary, $\parc \Sigma_n$) in $\r^3$ with respect to $F$, choose points $p_n\in\Sigma_n$ for each $n$, and assume that the following conditions hold:
\begin{enumerate}
\item [(i)]There exists a sequence of positive numbers $\{r_n\}$ with $r_n\rightarrow\infty$ such that the geodesic disks $D(p_n,r_n)\subset \Sigma_n$ centered at  $p_n$ and of radius $r_n$ are contained in the interior of $\Sigma_n$, i.e., $d(p_n,\partial\Sigma_n)\geq r_n$.
\item [(ii)] $\{p_n\}\rightarrow p_0$ for some $p_0\in\r^3$.
\item [(iii)] There exists $C>0$ such that $|\sigma_n(x)|\leq C$ for all $n\in\n$ and all $x\in\Sigma_n$, where $|\sigma_n|$ denotes the norm of the second fundamental form of $\Sigma_n$.
\item [(iv)] $H_n\rightarrow H_0\in\r$, where $H_n$ is the (constant) anisotropic mean curvature of $\Sigma_n$.
\end{enumerate} 
Then, for any $k\geq 2$, there exists a subsequence of $\{\Sigma_n\}$ that converges uniformly on compact sets in the $C^k$ topology to a complete immersed surface without boundary in $\R^3$, possibly non-connected, passing through $p_0$, with  bounded second fundamental form, and with constant anisotropic mean curvature equal to $H_0$.
\end{theorem}
\begin{proof}
By a well-known result about immersed surfaces with bounded second fundamental form in Riemannian $3$-manifolds (see e.g. Proposition 2.3 in \cite{RST}), it follows from (i) and (iii) that there exist positive constants $\delta,\mu>0$ that only depend on $C$ (and not on $H_n$ or $\Sigma_n$) such that for any $n$ sufficiently large, the following properties hold:
\begin{enumerate}
\item An open neighborhood of $p_n\in\Sigma_n$ is the graph of a function $v_n$ defined on the Euclidean disk $D_{2\delta}\subset T_{p_n}\Sigma_n$ centered at the origin, and of radius $2\delta$ .
\item The ${\mathcal C}^2$ norm of the function $v_n$ in $D_{2\delta}$ is at most $\mu/2$. 
\end{enumerate}

By passing to a subsequence if necessary, we may assume that the Gauss map images in $\S^2$ of the points $p_n$ converge to a unit vector $N_0\in\s^2$. Thus, after a change of Euclidean coordinates $(x,y,z)$ so that $p_0$ corresponds to the origin and $N_0=(0,0,1)$, we have that, for $n$ sufficiently large:
\begin{enumerate}
\item An open neighborhood $D_n$ of $p_n\in\Sigma_n$ can be seen as the graph of a function $u_n$ defined on the disk $B_{\delta}=\{(x,y)\in\r^2:\ x^2+y^2<\delta^2\}$.
\item The ${\mathcal C}^2$ norm of the function $u_n$ in $B_{\delta}$ is at most $\mu$. 
\end{enumerate}

Moreover, by \eqref{edp}, we see that the functions $u_n$ are solutions to the linear PDE $L_nu_n=H_n$, where
$$
L_nu=a_n(x,y)\ u_{xx}+b_n(x,y)\ u_{xy}+c_n(x,y)\ u_{yy}
$$
and we are denoting $a_n(x,y):=a((u_n)_x(x,y),(u_n)_y(x,y))$, etc.

By the second condition above, each function $u_n$ lies in ${\cal C}^{1,\alpha}(B_{\delta})$. So, in particular, all the functions $a_n,b_n,c_n$ are bounded in the ${\mathcal C}^{0,\alpha}(B_{\delta})$ norm. It follows then by the classical Schauder theory that for any positive number $\delta'<\delta$ there exists a constant $C'$ independent of $n$ so that $\|u_n\|_{{\cal C}^{2,\alpha}(B_{\delta'})}\leq C'$.

Therefore, the coefficients $a_n,b_n,c_n$ of the linear equation $L_nu_n=H_n$ are uniformly bounded in the ${\cal C}^{1,\alpha}(B_{\delta'})$ norm. By iterating this process, we obtain for each $\delta'\in (0,\delta)$ the existence of a positive constant $C''=C''(\delta')$ such that
$$
\|u_n\|_{{\cal C}^{k,\alpha}(B_{\delta'})}\leq C'',
$$
for $n$ sufficiently large.

Once here, a standard application of the Arzelà-Ascoli theorem shows that there exists a subsequence of $\{u_n\}$ that converges on the disk $B_{\delta'}$ with respect to the ${\cal C}^k$ topology to a solution $u$ to \eqref{edp}. That is, the graph $\Sigma_u$ of the function $u(x,y)$ has CAMC equal to $H_0$. By construction, it also passes through $p_0$ and, since $k\geq 2$, the norm of the second fundamental form of $\Sigma_u$ is bounded by $C$.

Consider now some point $(x_0,y_0)\in B_{\delta'}$ and let $q\in\Sigma_u$ be its image via $u(x,y)$. Since $u$ is a limit of the functions  $u_n$, the points $q_n=(x_0,y_0,u_n(x_0,y_0))\in\Sigma_n$ converge to $q$. Therefore, after passing to a subsequence if necessary, we can assume that the first statement of Theorem \ref{compacidad} holds. Thus, by repeating the above argument, but this time with respect to the points $q_n$ and $q$ we obtain an immersed surface  $\Sigma$ with CAMC $H_0$ that extends $\Sigma_u$ and is well defined as a graph over the disk of radius $\delta'$ centered at the origin of the tangent plane $T_q\Sigma$.

Once here, we may use again the first condition of Theorem \ref{compacidad} and a standard diagonal process to show that $\Sigma$ can be extended to be a complete surface with CAMC $H_0$, that contains $p$ and whose norm of the second fundamental form is bounded by $C$. Moreover, by construction, such surface is a limit in the $\cC^k$ topology on compact sets of the sequence $\{\Sigma_n\}$, as we wished to show.
\end{proof}

\begin{remark}
Let $\kappa_1,\kappa_2$ denote the principal curvatures of an immersed surface $\Sigma$ in $\R^3$, and let $\lambda_1,\lambda_2$ denote its anisotropic principal cuvatures. Then, it is easy to see that the function $\kappa_1^2 + \kappa_2^2$ is bounded on $\Sigma$ if and only if $\landa_1^2 + \landa_2^2$ is bounded. Indeed, this is an immediate consequence of \eqref{1}, since the endomorphism $\cS:=d\eta^{-1}$ appearing there is (up to sign) the Weingarten endomorphism of an ovaloid of $\R^3$ (specifically, of the Wulff shape $\cW$).

As a result, the condition (iii) in Theorem \ref{compacidad} can be replaced by the existence of a constant $d>0$ such that the norms of the anisotropic Weingarten endomorphism $A_n$ of  $\Sigma_n$ satisfy
$
|(A_n)_x|\leq d$ for every $n\in\n$ and every $x\in\Sigma_n
$.
\end{remark}

In 1961, H.B. Jenkins proved in \cite{J} that any complete anisotropic minimal surface whose Gauss map image omits a spherical disk of $\S^2$ must be a plane. We will next use Jenkins' theorem to prove that if the Gauss map image of a complete surface $\Sigma$ with CAMC omits a spherical disk, then $\Sigma$ has bounded second fundamental form. As a matter of fact, we will prove the following more general estimate:

\begin{theorem}\label{curvatura}
Let $h,\rho,d$ be positive constants. Then, there exists a constant $C=C(h,\rho,d)$ such that the following assertion holds:

Let $\Sigma$ be a complete surface in $\R^3$, possibly with boundary, and with CAMC equal to $H\in \R$. Assume:
\begin{enumerate}
\item[(i)] $|H|\leq h$.
\item[(ii)] The Gauss map image $N(\Sigma)\subset\s^2$ of $\Sigma$ omits a spherical disk of radius $\rho$.
\end{enumerate}
Then, for any $p\in\Sigma$ with $d_{\Sigma}(p,\partial\Sigma)\geq d$ it holds
$$
|\sigma_{\Sigma}(p)|\leq C.
$$
Here, $d_{\Sigma}$ and $|\sigma_{\Sigma}|$ denote, respectively, the intrinsic distance in $\Sigma$ and the norm of the second fundamental form of $\Sigma$.
\end{theorem}
\begin{proof}
We proceed arguing by contradicion. If the statement of Theorem \ref{curvatura} was not true, there would exist a sequence of complete immersed surfaces $f_n:\Sigma_n\fl\r^3$, possibly with boundary, with CAMC of values $H_n$, which satisfy properties (i), (ii) above, and points $p_n\in\Sigma_n$ such that $d_{\Sigma_n}(p_n,\partial\Sigma_n)\geq d$ and $|\sigma_{\Sigma_n}(p_n)|>n$.

Take $c_n\in\s^2$ such that the Gauss map image of $f_n$ omits the geodesic disk of $\S^2$ centered at $c_n$ and of radius $\rho$. By compactness of $\S^2$, the sequence $\{c_n\}$ has some accumulation point $c_0\in\s^2$. So, passing to a subsequence if necessary, we may assume that the Gauss map image of all the immersions $f_n$ omit the same geodesic disk of $\S^2$ centered at $c_0$ and of radius $\rho/2$.

Let $D_n=D_{\Sigma_n}(p_n,d/2)$ denote the intrinsic compact metric disk in $\Sigma_n$ centered at $p_n$ and of radius $d/2$; note that $D_n$ is at a positive distance from $\partial\Sigma_n$. Let $q_n$ denote the maximum in $D_n$ of the function
$$
h_n(q)=|\sigma_{\Sigma_n}(q)|d_{\Sigma_n}(q,\partial D_n),\qquad q\in D_n.
$$

Since $h_n$ vanishes on $\partial D_n$, it is clear that $q_n$ lies in the interior of $D_n$. Consider now $\lambda_n=|\sigma_{\Sigma_n}(q_n)|$ and $r_n=d_{\Sigma_n}(q_n,\partial D_n)$. Then,
\begin{equation}\label{lann}
\lambda_n r_n=|\sigma_{\Sigma_n}(q_n)| d_{\Sigma_n}(q_n,\partial D_n)=h_n(q_n)\geq h_n(p_n)>n\ \frac{d}{2}.
\end{equation}
In particular, $\{\lambda_n\}\rightarrow \infty$ as $n \rightarrow \infty$. Let us also observe that, if we denote $\hat{D}_n=D_{\Sigma_n}(q_n,r_n/2)\subset D_n$, then for any $z_n\in \hat{D}_n$ it holds
\begin{equation}\label{destr}
d_{\Sigma_n}(q_n,\partial D_n)\leq 2 d_{\Sigma_n}(z_n,\partial D_n).
\end{equation}
Consider next the immersions $g_n:\hat{D}_n\fl\r^3$ given by restricting to the disks $\hat{D}_n\subset\Sigma_n$ the immersions $\landa_n f_n$. Then, we obtain from 
\eqref{destr} the following estimate for the norm of the second fundamental form $\hat{\sigma}_n$ of $g_n$, at any point $z_n\in\hat{D}_n$:
\begin{equation}\label{unisec}
|\hat{\sigma}_{n} (z_n)|= \frac{|\sigma_{\Sigma_n}(z_n)|}{\lambda_n} =\frac{h_n(z_n)}{\lambda_n d_{\Sigma_n}(z_n,\partial D_n)}\leq\frac{h_n(q_n)}{\lambda_n d_{\Sigma_n}(z_n,\partial D_n)}=\frac{d_{\Sigma_n}(q_n,\partial D_n)}{d_{\Sigma_n}(z_n,\partial D_n)}\leq 2.
\end{equation}
This shows that the norms of the second fundamental forms of the immersions $g_n$ are uniformly bounded, and moreover, that $|\hat{\sigma}_n(q_n)|=1$. Also, note that by \eqref{lann}, the radii of the disks $\hat{D}_n$ with respect to the metric induced by $g_n$ diverge to infinity.

Up to a translation, we can assume that $g_n(q_n)$ is the origin of $\r^3$. Also, up to taking a subsequence, we may assume that the Gauss map images of $g_n$ at $q_n$ converge to some unit vector $N_0\in \S^2$. 
Let us choose canonical Euclidean coordinates $(x,y,z)$ so that $N_0=(0,0,1)$.

Once here, we will use a similar argument to the one of Theorem \ref{compacidad} in order to prove that a subsequence of the immersions $g_n:\hat{D}_n\fl\r^3$ converges uniformly on compact sets to a complete immersion with vanishing anisotropic mean curvature.

First, from \cite[Proposition 2.3]{RST} and arguing as in Theorem \ref{compacidad}, we obtain the existence of positive constants $\delta_0,\mu$ (that do not depend on $n$) with the property that for any $n$ large enough, a neighborhood in $g_n(\hat{D}_n)$ of the origin is given by the graph $z=u_n(x,y)$ of a function $u_n$ defined on the disk $B_{\delta_0}\subset\r^2$ centered at the origin and of radius $\delta_0$, with $\|u_n\|_{C^2(B_{\delta_0})} \leq \mu$.

Since the immersions $g_n$ have CAMC of value $H_n/\lambda_n$ and $|H_n|\leq h$, it follows that their anisotropic mean curvatures converge to zero. In this way, we can repeat the argument of Theorem \ref{compacidad} to deduce that the functions $u_n$ converge in the $\mathcal{C}^2(B_{\delta'_0})$-topology (with $0<\delta'_0<\delta_0$) to a smooth function $u_0$ whose graph $\Sigma_0$ has zero anisotropic mean curvature. Moreover,  $\Sigma_0$ can be globally extended to a complete minimal anisotropic surface $\Sigma$ that, by construction, is a limit in the $\mathcal{C}^2$-topology on compact sets of the sequence $\{g_n(\hat{D}_n)\}$.

Since the norm of the second fundamental form of $g_n(\hat{D}_n)$ is equal to $1$ at the origin for all $n$, the same happens to $\Sigma$. On the other hand, the Gauss map image of $\Sigma$ omits the geodesic disk centered at $c_0$ and of radius $\rho/2$, since this happens for all immersions $g_n$. This implies by Jenkins' theorem \cite{J} that $\Sigma$ must be a plane. But this contradicts that the norm of the second fundamental form of $\Sigma$ at the origin is equal to $1$. This contradiction proves Theorem \ref{curvatura}.
\end{proof}
\section{Characterization of planes, cylinders and Wulff shapes}\label{sec:3}
The present section will be mostly devoted to prove the following key result:
\begin{theorem}\label{multigrafo}
Any complete multigraph with constant anisotropic mean curvature is a plane.
\end{theorem}
\begin{proof}
By Jenkins' theorem \cite{J}, a complete multigraph with zero anisotropic mean cuvature is a plane. So, to prove Theorem \ref{multigrafo} it suffices to check that there are no complete multigraphs with non-zero CAMC.

We will argue by contradiction. So, from now on, $\Sigma$ will denote a complete multigraph with CAMC of value $H_0\neq 0$. Up to a homothety of $\R^3$, we will assume $H_0=-1$. Moreover, we will fix Euclidean coordinates $(x,y,z)$ in $\R^3$ so that $\Sigma$ is a multigraph with respect to the $z=0$ plane, and so that the third coordinate $N_3$ of the unit normal $N$ to $\Sigma$ is negative at every point. We will let $\pi:\r^3\fl\r^2$ denote the vertical projection in $\R^3$.

For any given point $p\in\Sigma$, since $\Sigma$ is a multigraph, there exists a neighborhood $U\subset\Sigma$ of $p$ that is the graph $z=u(x,y)$ of a function $u$ defined on the disk $B(\pi(p),r)$ of $\r^2$ centered at $\pi(p)$ and of some radius $r>0$. This radius $r$ cannot be larger than $2d_{\cW}$, where $d_{\cW}$ is the diameter of the Wulff shape. Indeed, if $r>2d_{\cal W}$, let $u_m\in \R$ be the maximum value of $u$ on the closed disk  $\overline{B}(\pi(p),2d_{\cal W})$. Note that the dilation of ratio $2$ of the Wulff shape transforms $\cW$ into a surface with CAMC $-1$, that we will denote by $2{\cal W}$. In this way, we can translate $2{\cal W}$ in $\R^3$ so that it is placed over $\overline{B}(\pi(p),2d_{\cal W})$, at a height greater than $u_m$, and then translate $2{\cal W}$ downwards until reaching a first (interior) contact point with $U$; this contradicts the maximum principle.

In the next paragraphs, we fix some notation that will be used in the rest of the proof. 

Given any $p\in \Sigma$, we will denote by $r_0=r_0(p)\in (0,2d_{\cal W}]$ the largest value of the radius $r$ for which the function $u$ above can be extended to the open disk $B(\pi(p),r_0)$. We will also let  $q_0\in \partial B(\pi(p),r_0)$ be a point for which the function $u$ cannot be extended to a neighborhood of $q_0$.

By Theorem \ref{curvatura}, we have that the norm of the second fundamental form of $\Sigma$ is uniformly bounded. Therefore, there exists some $\delta>0$, that will be considered fixed from now on, with the following property: any $p\in\Sigma$ has a neighborhood ${\cal U}_p\subset \Sigma$ that is a graph over the disk $B_p(\delta)\subset T_p\Sigma$ centered at the origin and of radius $\delta$ of its tangent plane at $p$ (see \cite[Proposition 2.3]{RST}). Let ${\cal U}_p^v$ be the vertical translation of the neighborhood ${\cal U}_p$ that takes $p$ to $(\pi(p),0)$, i.e. $\cU_p^v =\cU_p - (0,0,p_3)$, where $p=(p_1,p_2,p_3)$.

Finally, given a point $q\in \R^2$ and a cylinder ${\cal C}$ with CAMC $-1$ that passes through $(q,0)$, we will let $C_q$ be the neighborhood of $(q,0)$ in ${\cal C}$ that is a graph over the disk centered at the origin and of radius $\delta$ of $T_{(q,0)} \mathcal{C}$.

With these notations and comments in mind, we will start by proving the following claim:

\begin{assertion}\label{ass:1}
Let $p\in \Sigma$ so that $\Sigma$ can be seen locally around $p$ as a graph $z=u(x,y)$ over a disk $B(\pi(p),r_0)\subset \r^2$, and so that there exists $q_0\in \parc B(\pi(p),r_0)$ for which $u$ cannot be extended to a neighborhood of $q_0$.

Then, for any sequence $\{q_n\}\subset B(\pi(p),r_0)$ converging to $q_0$, it holds that the translated graphs $\cU_{(q_n,u(q_n))}^v$ converge in the $C^2$-topology to the neighborhood $C_{q_0}$ of $(q_0,0)$ of a cylinder $\mathcal{C}$ with CAMC equal to $-1$ with respect to its exterior unit normal $N_C$, that passes through $(q_0,0)$, and such that $N_C (q_0,0)$ is collinear with the horizontal vector $(\pi(p)-q_0,0)$.
\end{assertion}
\begin{proof}
Let $N_3$ denote the third coordinate of the unit normal $N$ of $\Sigma$. First of all, let us see that $\{N_3(p_n)\}\rightarrow0$, where $p_n:=(q_n,u(q_n))$. Indeed, if this was not the case, there would exist a subsequence $\{q_n\}\rightarrow q_0$ with $\{N_3(p_n)\}\rightarrow N_0\in[-1,0)$. Since ${\cal U}_{p_n}$ is a graph over a disk in $T_{p_n}\Sigma$ of fixed radius $\delta>0$, and $N_0\neq 0$, then for $n$ sufficiently large there exists a fixed $\varepsilon >0$ and a neighborhood ${\cal V}_{p_n}\subset\Sigma$ of $p_n$ that can be seen as a vertical graph over $B(q_n,\varepsilon)\subset \R^2$. This contradicts the fact that $u$ cannot be extended to a neighborhood of $q_0$, choosing $q_n$ sufficiently close to $q_0$. Thus, $\{N_3(p_n)\}\rightarrow0$.

Take now a subsequence of $\{q_n\}$ so that $\{N(p_n)\}$ converges to some unit vector $v_0\in \S^2$; this subsequence exists by compactness of $\S^2$. Since $\{N_3(p_n)\}\rightarrow0$, $v_0$ is a horizontal vector. In these conditions, using the ideas in the proof of Theorem \ref{compacidad}, it is clear that, up to a subsequence, ${\cal U}_{p_n}^v$ converges in the ${\cal C}^2$-topology to a (non-complete) surface $S$ with CMAC equal to $-1$, that is a graph over its tangent plane at $(q_0,0)$, and whose unit normal at that point is $v_0$. But once here we can note that the third coordinate of the unit normal of $S$ is non-positive (since it is a limit of vertical graphs), and vanishes at $(q_0,0)$. By a standard application of the maximum principle, we deduce then that this third coordinate vanishes identically on $S$ (see e.g. \cite{KP}). In this way, $S$ is contained in a cylinder $\mathcal{C}=\Gamma\times \R$ with CAMC equal to $-1$, and whose exterior unit normal at $(q_0,0)$ is $v_0$. Thus, $S\subset{\cal C}$; as a matter of fact, $S=C_{q_0}$, where here $C_{q_0}$ denotes the $\delta$-neighborhood of $(q_0,0)$ in $\mathcal{C}$, as explained prior to the statement of Assertion \ref{ass:1}. 

Let us next show that $v_0$ is collinear with $(\pi(p)-q_0,0)$. Consider the planes $Q:=\{v_0\}^{\perp}$ and $P:=\{(\pi(p)-q_0,0)\}^{\perp}$, and assume that $P\neq Q$. Then, since the cylinder $\mathcal{C}=\Gamma\times \R$ is tangent to $Q$ at $(q_0,0)$, any open arc of the base curve $\Gamma\subset \R^2$ that contains $q_0$ intersects $B(\pi(p),r_0)$.

Let $a_0\in B(\pi(p),r_0)$ so that $(a_0,0)\in C_{q_0}$. Then, by the convergence of ${\cal U}^v_{p_n}$ to $C_{q_0}$, there exist $b_n\in {\cal U}_{p_n}^v$ with $\{b_n\}\rightarrow (a_0,0)$ and so that $(\pi(b_n),u(\pi(b_n)))$ lies in the graph of $u$. Note that the tangent planes to $b_n\in {\cal U}_{p_n}^v$ become vertical. Hence, $|{\rm grad}\,u(\pi(b_n))|\rightarrow\infty$, which is impossible since $\{\pi(b_n)\}\rightarrow a_0$ and $u$ is well defined around $a_0$. This contradiction proves $P=Q$. Note that by uniqueness of the limit, the convergence of the $\{{\cal U}_{p_n}^v\}$ to $C_{q_0}$ we have just proved is global, i.e. the whole sequence converges and not just a subsequence of it. This finishes the proof of Assertion \ref{ass:1}. 
\end{proof}

It should be noted that there are two cylinders $\mathcal{C}=\Gamma\times \R$ that satisfy the conditions stated in Assertion \ref{ass:1}; they have opposite unit normals at $(q_0,0)$, and differ by a translation in $\R^3$. The next assertion is helpful in determining which of these two cylinders appears in the limit process described in Assertion \ref{ass:1}.

\begin{assertion}\label{ass:2}
In the conditions of Assertion \ref{ass:1}, let $\gamma_0(t)=(1-t)q_0+t \pi(p)$, $t\in(0,1]$, join $q_0$ and $\pi(p)$. Then, the function $u_0(t)=u(\gamma_0(t))$ satisfies that $\lim_{t\rightarrow 0}u_0(t)=\infty$ (resp. $-\8$) if $\gamma_0(t)$ lies locally in the convex (resp. concave) side of $\Gamma$ at $q_0$.
\end{assertion}
\begin{proof}
The function $u_0(t)$ is strictly monotonic for $t$ close to $0$, since by Assertion \ref{ass:1}, the unit tangent vector to the curve $(\gamma_0(t),u(\gamma_0(t)))$ has limit $(0,0,\pm1)$ as $t\rightarrow0$.

Let $h_0:=\lim_{t\rightarrow0}u_0(t)\in\r\cup\{-\infty,\infty\}$. In case $h_0\in\r$, the length of the curve $(\gamma_0(t),u(\gamma_0(t)))$ is finite, by the previous monotonicity property of $u_0(t)$. Thus, by completeness of $\Sigma$, we have $(q_0,h_0)\in\Sigma$, and moreover, the tangent plane to $\Sigma$ at this point  $(q_0,h_0)$ is a vertical plane, again by Assertion \ref{ass:1}. This is impossible, since $\Sigma$ is a multigraph.

Thus, $h_0=\pm \8$. Finally, since the unit normal to $\Sigma$ points downwards (i.e. $N_3<0$), we deduce that if $\lim_{t\rightarrow0}u_0(t)=\infty$ (resp. $-\infty$) then the (horizontal) limit unit normal of $\Sigma$ along $(\gamma_0(t),u(\gamma_0(t)))$ points in the direction of the vector $(q_0-\pi(p),0)$ (resp. $(\pi(p)-q_0,0)$). This proves Assertion \ref{ass:2}, taking into account that the limit cylinder $\Gamma\times \R$ is oriented with respect to its outer unit normal.
\end{proof}

For the next assertion, let $\Gamma(s)$ be an arc-length parametrization of $\Gamma$, with $\Gamma(0)=q_0$. Then, the neighborhood $C_{q_0}\subset\Gamma\times\r$ of the point $(q_0,0)$, projects to an open arc $\pi(C_{q_0})$ of $\Gamma$ that contains $\Gamma([-\delta,\delta])$. We then define the subset of $\r^2$ 
\begin{equation}\label{oep}
{\cal O}_{\varepsilon}=\{\Gamma(s)+t \,n_{\Gamma}(s):\ s\in[-\delta,\delta],t\in(0,\varepsilon)\},
\end{equation}
where $n_{\Gamma}(s)$ is the unit normal of $\Gamma(s)$ that, for $s=0$, points in the direction $\pi(p)-q_0$.

Recall that $u(x,y)$ is defined on $B(\pi(p),r_0)$ and cannot be extended across $q_0\in \parc B(\pi(p),r_0)$. With the previous definitions in mind, we will next prove an extension property of $u$ outside $B(\pi(p),r_0)$.

\begin{assertion}\label{ass:3}
In the above conditions, the graph $u(x,y)$ extends smoothly to $B(\pi(p),r_0)\cup {\cal O}_{\varepsilon}$ for some $\varepsilon>0$. Moreover this extension satisfies that $u(q)$ diverges to $\pm \8$ when $q\in{\cal O}_{\varepsilon}$ approaches $\Gamma.$
\end{assertion}
\begin{proof}
Given $t_0\in(0,1]$, let us define the open set $\Sigma_{t_0}\subset \Sigma$ given by
\begin{equation}\label{sito}
\Sigma_{t_0}=\bigcup_{0<t<t_0}{\cal U}_{(\gamma_0(t),u(\gamma_0(t)))},
\end{equation}
which is a connected neighborhood of the curve $\{(\gamma_0(t),u(\gamma_0(t))):\ 0<t< t_0\}\subset\Sigma.$ 

For each $s\in [-\delta,\delta ]$, let $P(s)$ be the vertical plane normal to $\Gamma$ that passes through $\Gamma(s)$. Recall that we proved in Assertion \ref{ass:1} that  
$$
{\cal U}^v_{(\gamma_0(t),u(\gamma_0(t)))}\fl C_{q_0}\subset\Gamma\times\r,\qquad \mbox{ when }
t\rightarrow0,
$$
in the ${\cal C}^2$ topology. In particular, this shows that the projection $\pi(\Sigma_{t_0})\subset \R^2$ of $\Sigma_{t_0}$ in \eqref{sito} contains some open set $\mathcal{O}_{\ep}$ as in \eqref{oep}. Also, it shows that there is some $t_0>0$ such that $P(s)$ intersects $\Sigma_{t_0}$ transversely for all $s\in[-\delta,\delta]$. Observe that all points in $\Sigma_{t_0}\cap P(0)$ lie in the curve $(\gamma_0(t),u(\gamma_0(t)))$, and in particular $\Sigma_{t_0}\cap P(0)$ is a connected graphical curve. In the same way, by transversality and the definition of $\Sigma_{t_0}$, it follows that there is some $t_0>0$ and some $\ep>0$ such that for each $s\in [-\delta,\delta]$, $\Sigma_{t_0}\cap P(s)$ is a unique curve, given as a graph over a segment in $\R^2$ of the form $\Gamma(s) + t n(s)$, where $t$ varies in an interval $I_s$ that contains $(0,\ep)$.

These properties show that $\Sigma_{t_0}$ is a graph when restricted to $\{q\in \Sigma_{t_0} : \pi(q)\in\mathcal{O}_{\ep}\}$. In particular, this proves that $u$ can be extended as a graph to $B(\pi(p),r_0)\cup {\cal O}_{\varepsilon}$.

Let us next prove that there exists some $t_0>0$ such that $\Sigma_{t_0}$ does not intersect $\Gamma\times \R$. To this respect, note that by Assertions \ref{ass:1} and \ref{ass:2}, and the definition of $\Sigma_{t_0}$, the curve $\Sigma_{t_0}\cap P(s)$ is asymptotic to the cylinder $\Gamma\times\r$ at infinity, but in principle it could intersect it.

In order to prove that $\Sigma_{t_0}$ does not intersect $\Gamma\times \R$ we suppose next, arguing by contradiction, that there is some $s_0\in(0,\delta]$ for which the curve $\Sigma_{t_0}\cap P(s_0)$ crosses the cylinder $\Gamma\times\r$ (the argument for $s_0\in [-\delta,0)$ is analogous). Then, since $\Sigma$ is a multigraph, the curve $\Sigma_{t_0}\cap P(s_1)$ also crosses $\Gamma\times\r$ for any $s_1<s_0$ sufficiently close to $s_0$.

This shows that there are two possible situations. Either $\Sigma_{t_0}\cap P(s)$ never intersects $\Gamma\times\r$, or else at the smallest value of $s\in (0,\delta]$ for which $\Sigma_{t_0}\cap P(s)$ intersects $\Gamma\times\r$, it happens that $\Sigma_{t_0}\cap P(s)$ does not cross $\Gamma\times\r$. But this second situation is impossible, since in that case, for $p_0\in \Sigma_{t_0}\cap P(s)\cap (\Gamma\times\r)$, the tangent plane of $\Sigma$ at $p_0$ would be vertical, what contradicts that $\Sigma$ is a multigraph.

Consequently, $\Sigma_{t_0}$ does not intersect $\Gamma\times \R$, for some $t_0>0$. This fact together with the asymptotic convergence of each curve $\Sigma_{t_0}\cap P(s)$ to the cylinder $\Gamma \times \R$ proves the asymptotic behavior in the statement, and completes the proof of Assertion \ref{ass:3}.
\end{proof}

We next make a continuation argument. Recall that the point $p$ was arbitrarily chosen on $\Sigma$. Thus, by choosing $p\in \Sigma$ so that its projection $\pi(p)$ lies sufficiently close to $\Gamma(\delta/2)$ and inside the half-line $\{\Gamma(\delta/2)+tn_{\Gamma}(\delta/2) : t>0\}$, we clearly see that 
the graph $u$ could also be extended along $\Gamma$ to the set
\begin{equation}\label{extens}
\{\Gamma(s)+t \,n_{\Gamma}(s):\ s\in[-\delta/2,3\delta/2],t\in(0,\varepsilon')\},
\end{equation}
for some $\varepsilon'>0$. This process can be continued. Specifically, assume that $\Gamma(s)$ is an injective parametrization of $\Gamma$ on an interval $(a,b]$, with $a<0<b$ and $\lim_{s\to a^+} \Gamma(s)=\Gamma(b)$ (recall that $\Gamma(0)= q_0$). Then, by the process above, there exists $\ep>0$ such that the function $u(x,y)$ can be smoothly extended to the open simply connected set
\begin{equation}
\label{entorno}
\{\Gamma(s)+t \,n_{\Gamma}(s):\ s\in(a,b),t\in(0,\varepsilon)\}.
\end{equation}

It is important to observe here that, since $\Sigma$ is a multigraph (not necessarily a graph), this extension cannot be carried out in a continuous way to the open annulus given by \eqref{entorno}, but this time with $s\in(a,b]$. Specifically, the extensions of $u$ along $\Gamma$ for positive values of $s$ and for negative values of $s$ might not glue together continuously as $s$ reaches the limit values $a$ and $b$.

Up to this moment, the proof has been following closely the related theorem in \cite{Hauswirth}. From now on, the argument is different.

Recall that, in the arguments above, there are two possible orientations for the limit cylinder $\Gamma\times \R$, as explained in Assertion \ref{ass:1}. So, to end up the proof we will distinguish two different cases, depending on the orientation of $n_{\Gamma}$ with respect to $\Gamma$ in the above construction, i.e. depending on whether $\Sigma$ lies locally on the convex or the concave part of the cylinder $\Gamma\times\r$ in the previous argument.

{\bf Case 1:} \emph{There exists $p\in \Sigma$ such that, for its corresponding point $q_0\in\partial B(\pi(p),r_0)$, the vector $n_{\Gamma}(0)$ in \eqref{oep} is the interior unit normal to $\Gamma$ at $q_0$.}

In that situation, we know by the previous discussion that $\Sigma$ is the graph of a function $u$ over an open set of the form \eqref{entorno} in the interior of $\Gamma$. Let us show that $\Sigma$ must actually be a graph over the whole region bounded by $\Gamma$.

For that, let  $\tilde{q}_0$ be the other point in $\Gamma$ with tangent line parallel to the tangent line of $\Gamma$ at $q_0$, and let
$$
\Gamma_{\lambda}:=\lambda(\tilde{q}_0-q_0)+\Gamma,\qquad 0<\lambda\leq 1,
$$
denote the Euclidean translation of $\Gamma$ with translation vector $\lambda(\tilde{q}_0-q_0)$. Let $\Omega_{\lambda}$  denote the planar (open) domain bounded by  $\Gamma$ and $\Gamma_{\lambda}$ that contains the open segment from $q_0$ to $q_0+\lambda(\tilde{q}_0-q_0)$. Observe that $\Omega_1$ coincides with the interior region bounded by $\Gamma$.

Since $u$ can be extended to an open subset of $\Omega_1$ of the form \eqref{entorno}, we see that there exists $\lambda>0$ such that $u$ is well defined in $\Omega_\lambda$. Let $\landa_0$ be the supremum of the values $\lambda>0$ for which $u$ is well defined in $\Omega_{\landa}$. 

Assume that $\lambda_0<1$. Then, there would exist some $q_1\in\Gamma_{\lambda_0}$ lying in the interior of $\Gamma$, and such that the restriction of $u$ to $\Omega_{\landa_0}$ cannot be extended across $q_1$.
Consider next a new point  $p_1$ of the graph of $u$ such that $B(\pi(p_1),r_1)\subset\Omega_{\lambda_0}$ and $q_1\in \parc B(\pi(p_1),r_1)$, for some $r_1>0$. Then, by Assertion \ref{ass:1} there exists a translation of the cylinder $\Gamma\times \R$ such that the graph of $u$ converges asymptotically to it as we approach $q_1$. Since the base curve of this translated cylinder cannot cross $\Omega_{\lambda_0}$ (by Assertion \ref{ass:3}), we conclude that this cylinder must be equal to $\Gamma_{\lambda_0}\times\r$. So, by Assertions \ref{ass:2} and \ref{ass:3}, and the extension process described in \eqref{extens}, the function $u$ converges to $-\8$ when we approach $\Gamma_{\landa_0}$ by points $q\in\Omega_{\lambda_0}$. But this contradicts the fact that $u$ is well defined in \eqref{entorno}, and in particular at the points of $\Gamma_{\lambda_0}$ that are sufficiently close to $\Gamma_{\lambda_0}\cap\Gamma$. This contradiction proves that $\landa_0=1$, and so $u$ is well defined as a graph in $\Omega_1$, i.e., in the interior region bounded by $\Gamma$. In particular, as $\Sigma$ is connected, we see that $\Sigma$ is the graph of a function $u$ over $\Omega_1$, as claimed. Moreover, $u\to +\8$ as we approach $\parc \Omega_1=\Gamma$, by the previous discussion. 

Let us next prove that this is not possible, by the maximum principle. Let $2\cW$ be the homothety of ratio $2$ of the Wulff shape; it has CAMC equal to $-1$ and, after an adequate translation, its vertical projection is equal to $2\overline{\Omega_1}$. Let $2\cW^-$ denote the set of points of $2\cW$ whose exterior unit normal does not point upwards; the projection of $2\cW^-$ is again $2\overline{\Omega_1}$.

Since $\pi(\Sigma)=\Omega_1$ and $u\to +\8$ as we approach $\parc \Omega_1$, it is clear that $2\cW^-$ lies strictly below $\Sigma$ after an adequate vertical translation. Thus, moving then $2\cW^-$ vertically upwards we will eventually reach an interior first contact point of $\Sigma$ with $2\cW^-$. This contradicts the maximum principle. Thus, Case 1 above cannot happen.

{\bf Case 2:} \emph{For every $p\in\Sigma$ and every $q_0\in\partial B(\pi(p),r_0)$ corresponding to it, the vector $n_{\Gamma}(0)$ in \eqref{oep} is the exterior unit normal to $\Gamma$ at $q_0$.}

Let us start by recalling our setting. Take $p\in\Sigma$, let $u(x,y)$ be the local function that parametrizes a neighborhood of $p$ as a graph, and let $r_0>0$ be such that $u$ is well defined in $B(\pi(p),r_0)$ but there is some $q_0\in\partial B(\pi(p),r_0)$ for which $u$ cannot be extended to a neighborhood of $q_0$. By previous arguments, there exists a closed convex planar curve $\Gamma_0$ that contains $q_0$ such that $\Gamma_0\times\r$ is a cylinder of CAMC $-1$ with respect to its outer normal, and so that  $u$ can be extended to an \emph{exterior} neighborhood of $q_0$ of the form \eqref{oep}; here we use the word \emph{exterior} to mean that, since in the present case $n_{\Gamma}(0)$ in \eqref{oep} is the exterior unit normal to $\Gamma_0$ at $q_0$, the domain of definition of $u$ around $q_0$ is contained in the exterior of the curve $\Gamma_0$. Also, by the extension process described before \eqref{entorno}, we know that we can extend $u$ to an exterior neighborhood of $\Gamma_0\backslash\{\tilde{q}_0\}$ like the one in \eqref{entorno}, where here $\tilde{q}_0$ denotes again the unique point in $\Gamma_0\backslash\{q_0\}$ with tangent line parallel to the one at $q_0$. We also know, by Assertions \ref{ass:2} and \ref{ass:3}, that $u\to -\8$ when we approach any point in $\Gamma_0\backslash\{\tilde{q}_0\}$.

Let $2\Gamma_0\subset\r^2$ denote the homothety of ratio $2$ of $\Gamma_0$, translated so that it is tangent to $\Gamma_0$ at $\tilde{q}_0$ and contains $\Gamma_0-\{\tilde{q}_0\}$ in its interior.

Let $H:[0,1]\times\s^1\fl\r^2$ be a smooth, one-to-one (homothopy) mapping given by the following properties (see Figure \ref{fig1}).

\begin{enumerate}
\item
For each $t\in [0,1]$, $H_t:=H(\cdot, t):\S^1\flecha \R^2$ is a regular parametrization of a curve $C_t\subset \R^2$ that is homothetic to $\Gamma_0$, in the sense that $C_t$ differs from $\Gamma_0$ by some homothety of $\R^2$ of ratio $\landa_t>0$, followed by some translation of $\R^2$.
 \item
$H_1(\S^1)=2\Gamma_0$.
 \item
$H_0(\S^1)$ is contained in $B(\pi(p_0),r_0)\cup \{q_0\}$, and it is tangent to $\parc B(\pi(p_0),r_0)$ at $q_0$.\end{enumerate}

\begin{figure}[h]
\begin{center}
\includegraphics[width=8cm]{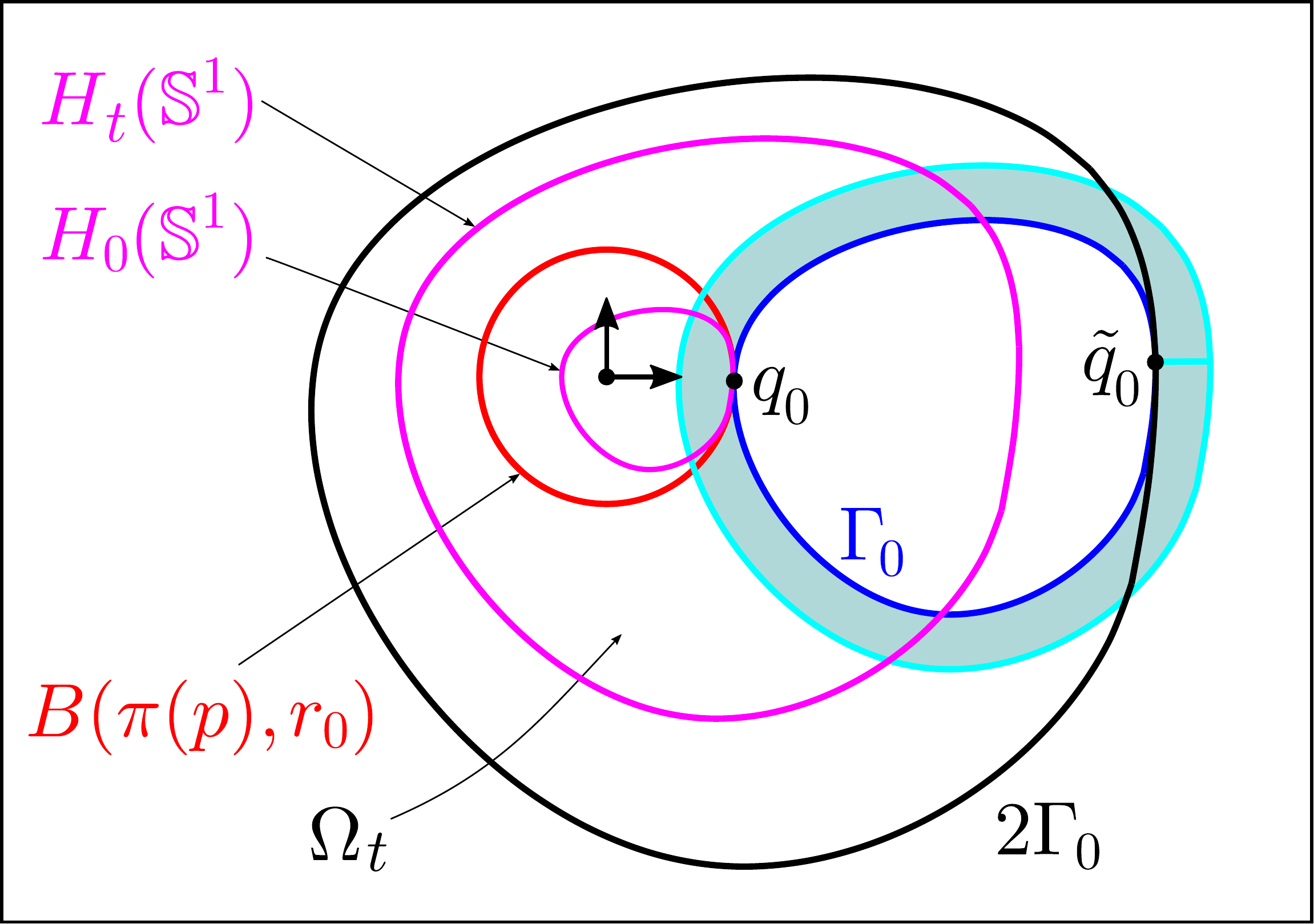}
\caption{Definition of the convex curve $H_t(\S^1)$ and the compact domain $\Omega_t$.} \label{fig1}
\end{center}
\end{figure}

For each $t\in[0,1]$, let $\Omega_t\subset\r^2$ be the compact convex domain bounded by $H_t(\S^1)$, and let $S_t\subset \Sigma$ be the connected component of $\Sigma\cap(\Omega_t\times\r)$ that contains $p$. Note that $S_{t_2} \subset S_{t_1}$ if $t_2<t_1$. Our next objective will be to prove that $S_1$ is a graph over some subset of $\Omega_1$, i.e. that $u$ can be extended to $\pi(S_1)\subset \Omega_1$.

To start, let $D_0$ denote the compact region of $\R^2$ bounded by $\Gamma_0$. Note that $\Omega_0\cap D_0=\{q_0\}$, and that $\Omega_t\cap D_0\neq \emptyset$, for all $t\in [0,1]$. By previous discussions, the function $u$ is well defined on $\tilde{\Omega}_0:=\Omega_0-\{q_0\}=\Omega_0-D_0$, and it holds that $u(q)\to -\8$ when $q\in \tilde{\Omega}_0$ approaches $q_0$.

We next analyze how this picture unfolds when we enlarge $\Omega_0$ to $\Omega_t$, for $t>0$. In order to do so, let $\mathcal{I}$ denote the set of values $t\in [0,1]$ for which:

\begin{enumerate}
\item[i)]
$S_t$ is the graph of (an extension of) $u$ over a subset $\tilde{\Omega}_t \subset \Omega_t$ given as
\begin{equation}\label{tildo}
\tilde{\Omega}_t=\Omega_t\backslash(D_0\cup\ldots\cup D_{m(t)}), 
\end{equation}
for some $m(t)\in\n\cup\{0\}$, where each $D_i\subset \R^2$ is the compact convex region bounded by some translation of the curve $\Gamma_0$, and all the $D_i\cap \Omega_t$ are pairwise disjoint (recall that $D_0$ is bounded by $\Gamma_0$). 
 \item[ii)]
$u(q)\to -\8$ when $q\in \tilde{\Omega}_t$ approaches $\parc D_0\cup \ldots \cup \parc D_m$.
\end{enumerate}

Note that in this definition, the disks $D_i$ are not contained in $\Omega_t$; they are only subject to the condition that $D_i\cap \Omega_t$ is non-empty (the intersection could be a single point, as it happens with the case $t=0$ explained above).

It is obvious by the previous discussion that $0\in{\cal I}$, and in that case $\tilde{\Omega}_0=\Omega_0-D_0$. Note that, since $S_{t_2}\subset S_{t_1}$ whenever $t_2<t_1$, it is also clear that if $t_1\in{\cal I}$ and $t_2\in[0,t_1]$, then $t_2\in{\cal I}$. Therefore, ${\cal I}$ is an interval of the form $[0,a)$ or $[0,a]$ for some $a\in(0,1]$. Moreover, the same domains $D_i$ appearing in the decomposition \eqref{tildo} of $\tilde{\Omega}_{t_2}$ will appear in the decomposition of $\tilde{\Omega}_{t_1}$, if $t_2<t_1$. In particular, the numbers $m(t)$ are non-decreasing with respect to $t$.

Assume next that ${\cal I}=[0,a]$, with $a\neq1$. Hence, there exists $m(a)\in\n\cup\{0\}$ so that $S_a$ is the graph of $u$ over
\begin{equation}\label{dominioss}
\tilde{\Omega}_a=\Omega_a\backslash(D_0\cup\ldots\cup D_{m(a)}).
\end{equation}
We want to prove that, for small values $\ep>0$, $S_{a+\ep}$ is a graph over $\Omega_{a+\ep}$ minus the same domains $D_0,\dots, D_{m(a)}$. Note that for any such $D_i$ we have that $\parc D_i\cap \parc \Omega_a$ either consists of two points, or $\parc D_i$ and $\parc \Omega_a$ are tangent at one point, and in that second situation they have opposite interior unit normals. Indeed, if their interior unit normals agreed at the intersection point, we would have $D_i\subset \Omega_a$ (it is impossible that $\Omega_a$ is contained in $D_i$, as we are in Case 2). But this condition implies that $D_i\subset \Omega_1$, and so $D_i\cap D_0\neq \emptyset$, since two translations of $\Gamma_0$ do not fit inside $2\Gamma_0$ without having an intersection point. And since $D_i\subset \Omega_a$, we have then that $D_i\cap D_0\cap \Omega_a\neq \emptyset$, which is a contradiction. 

So, once we have clarified the structure of each $\parc D_i\cap \parc \Omega_a$, the arguments of the first part of this proof show that $u$ can be extended to a small exterior tubular neighborhood of each curve $\parc D_i$, around their intersection points with $\parc \Omega_a$ (at most two points, for each such $D_i$). Note that we can choose such exterior tubular neighborhoods so small that they do not overlap each other (see Fig. \ref{fig2}). It is also clear that we can continue $u$ smoothly across the points in $\parc \Omega_a$ that do not lie in $\parc D_i$ for any $i$, since at those points $u$ is well defined. Therefore, for small values $\ep>0$, the domains $D_i$ are mutually disjoint inside $\Omega_{a+\ep}$, and $S_{a+\ep}$ is a graph over $\Omega_{a+\ep}\setminus(D_0\cup \cdots \cup D_{m(a)})$. In particular, this shows that there is some $\ep>0$ such that $[0,a+\ep)\subset \mathcal{I}$.

\begin{figure}[h]
\begin{center}
\includegraphics[width=8cm]{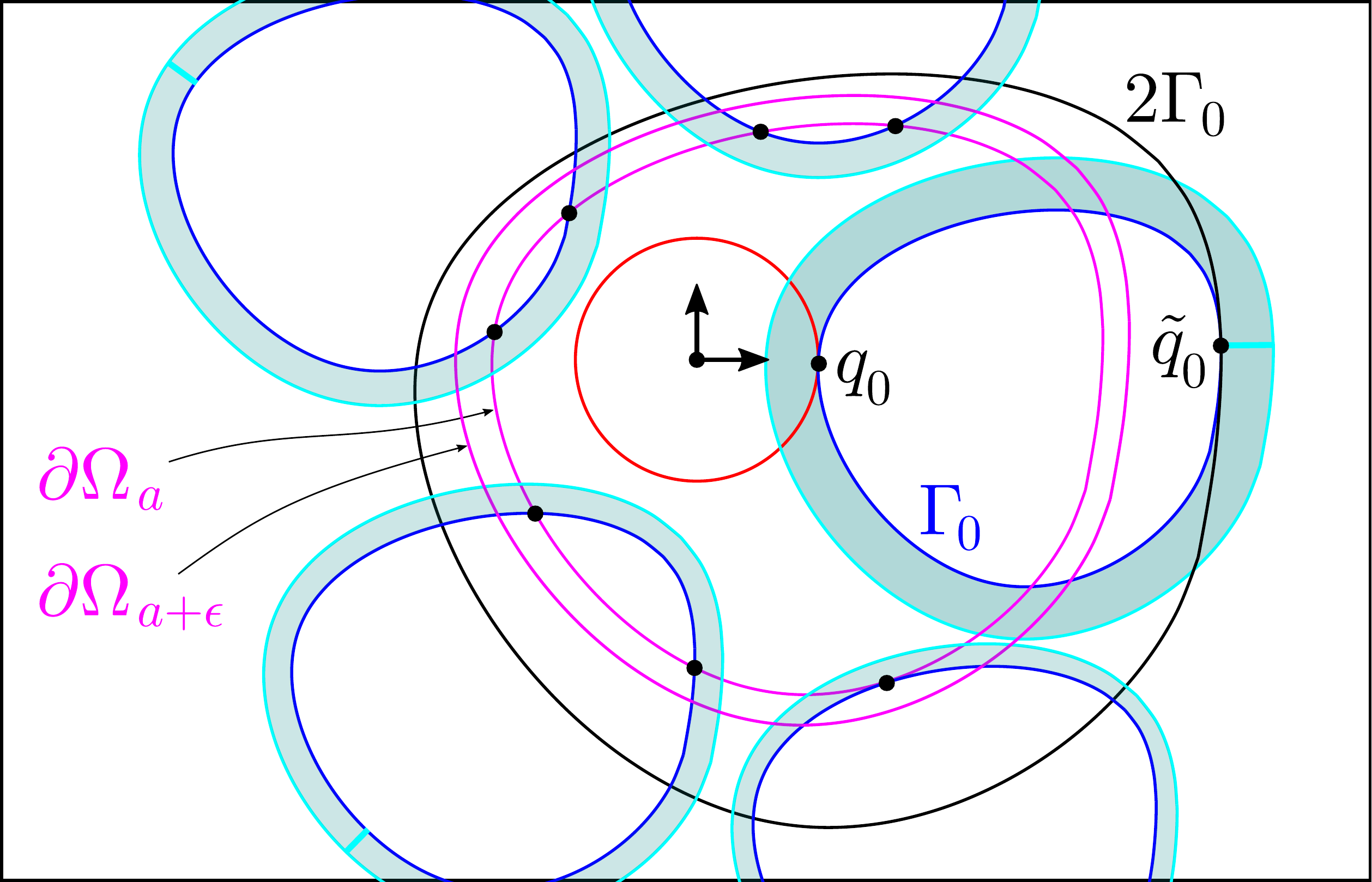}
\caption{Decomposition of the domain $\tilde{\Omega}_{a+\ep}$.} \label{fig2}
\end{center}
\end{figure}

Assume next that ${\cal I}=[0,a)$, with $0<a\leq1$. It is then clear that $S_a$ is a graph, since if two distinct points $p_1,p_2\in S_a$ satisfied $\pi(p_1)=\pi(p_2)$, and, since $\Sigma$ is a multigraph, there would exist $\tilde{p}_1,\tilde{p}_2\in S_b$ with $0<b<a$ such that $\pi(\tilde{p}_1)=\pi(\tilde{p}_2)$, and this would contradict that $b\in{\cal I}$. 

We will next show that the domain of $S_a$ is of the form explained in i), ii) above, what will prove that $a\in \mathcal{I}$. More specifically, we will prove that the domains $D_i$ appearing in the definition of $\tilde{\Omega}_a$ in \eqref{tildo} are the union of the domains $D_i$ which appear in the decomposition of $\tilde{\Omega}_b$, for all $b<a$ (we will show that there is only a finite number of such domains), and of a finite number of new domains  bounded by translations of $\Gamma_0$ that are tangent to $\parc \Omega_a$ on its concave side.

Given $b<a$, let us write 
\begin{equation}\label{decob}
\tilde{\Omega}_b=\Omega_b\backslash(D_1\cup\ldots\cup D_{m(b)}), 
\end{equation}
with $m(b)\in\n\cup\{0\}$. It was explained previously that, for any $t\in (b,a)$, the domains $D_i$ appearing in \eqref{decob} also appear in the decomposition \eqref{tildo} of $\tilde{\Omega}_t$ for that value $t$. In particular, all these $D_i\cap \Omega_t$ are disjoint, for all $t\in (b,a)$. Let us show that the domains $D_i\cap \Omega_a$ are also disjoint.

Assume first of all that all disks $D_i \neq D_0$ have points outside $\Omega_a$ (note that $D_0$ has points outside $\Omega_a$ except if $a=1$, in which case $D_0$ is tangent to $\Omega_1$ at $\tilde{q}_0$). In that case, again by the extension process in \eqref{extens}, the graph $u$ can be extended locally around any point in $\parc \Omega_a$ that belongs to some of the curves $\parc D_i$ to an exterior neighborhood of the point, so that $u\to -\8$ as it approaches $\parc D_i$. In particular, all the $D_i\cap \Omega_a$ (including $D_0$) are disjoint.

Assume next that the previous condition does not hold, i.e. there is some $D_j\neq D_0$ that is contained in $\Omega_a$. Note that $D_j\cap D_0\neq \emptyset$ since, again, two translated copies of $\Gamma_0$ that lie inside the compact region $\Omega_1$ bounded by $2\Gamma_0$ cannot be disjoint. Thus, $D_j\cap D_0$ is either one point at which $\parc D_j$ and $\parc D_0$ are externally tangent, or $\parc D_j\cap \parc D_0$ is a pair of points. In both cases, $\parc D_j\cap \parc D_0$ lies in $\parc \Omega_a$, since $D_j\cap D_0\cap \Omega_t$ is empty, for all $t<a$.

The case that $D_j\cap D_0$ is a single point in $\parc \Omega_a$ is impossible, due to the fact that, because $D_0$ and $D_j$ would be externally tangent at that point in this situation, this would force $D_0$ to lie in the exterior region of $\Omega_a$, and this is a contradiction with $D_0\cap \Omega_t\neq \emptyset$ for $t<a$.

On the other hand, if $\parc D_j\cap \parc D_0$ consists of two points in $\parc \Omega_a$, we would have that $\parc D_j\subset \Omega_a$ intersects $\parc \Omega_a$ \emph{tangentially} at two different points, and this is impossible unless $D_j=\Omega_a$, which cannot happen since $D_j\cap D_0\cap \Omega_t$ is empty for $t<a$.

This contradiction shows that the domains $D_i\cap \Omega_a$ are disjoint.

Consider next the map $t\mapsto m(t)$, which we know is non-decreasing. Let us show next that $m(t)$ is bounded as $t\to a^-$, i.e. that the total number of disks $D_i$ appearing in \eqref{decob} for all values $b\in (0,a)$ is finite.

Arguing by contradiction, assume that there exists a strictly increasing sequence
$\{t_n\}_n$ converging to $a$, and domains $D_{m(t_n)}$ (which arise in the decomposition of $\tilde{\Omega}_{t_n}$),
each of them bounded by a translation $\Gamma_{t_n}$ of $\Gamma_0$, and all of them pairwise disjoint inside $\Omega_a$ (by the argument above). We also can
suppose that the $\Gamma_{t_n}$ are tangent to $\parc \Omega_{t_n}$ on its concave side, and that 
$D_{m(t_n)}\cap \parc \Omega_a$ is a compact arc $J_n\subset \parc \Omega_a$, with endpoints $\{q_1^n,q_2^n\}$. These arcs are disjoint, and of arbitrary small length, taking $n$ sufficiently large.
In particular there exists an accumulation point $q^*\in \parc \Omega_a-\cup_{n\in \N} J_n$ of the sequence of pairs $\{q_1^n,q_ 2^n\}_n$.

Let $v^*\in \R^2$ denote the inner unit normal of $\parc \Omega_a$ at $q^*$, and consider the curve $\gamma^* :(0,1]\flecha \R^2$ given by $\gamma^*(s)=q^* +s v^*$, (see Figure \ref{fig3}). Note that $u$ is well defined along $\gamma^*$, for small values $s > 0$. 

\begin{figure}[h]
\begin{center}
\includegraphics[width=8cm]{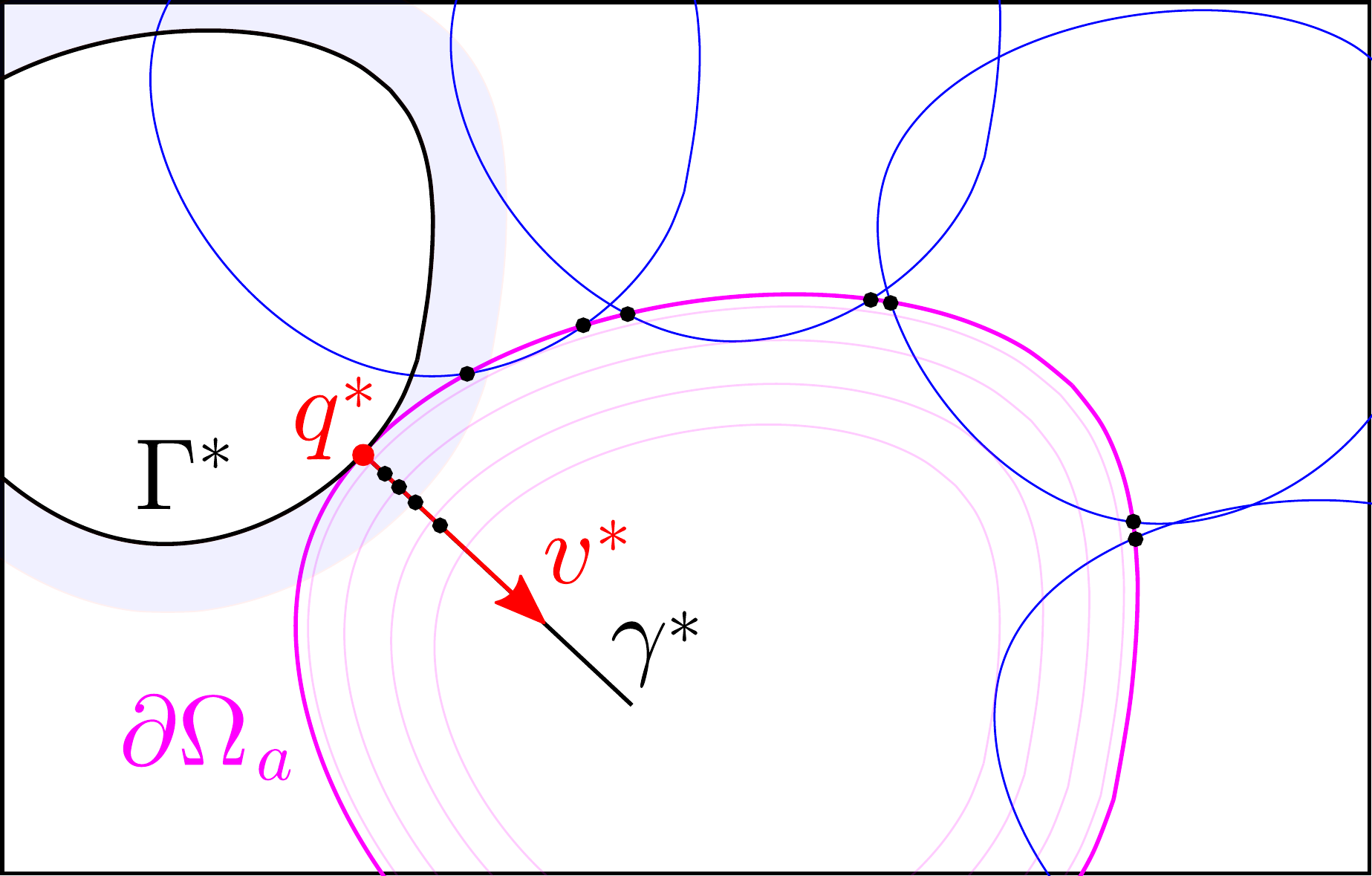}
\caption{The segment $\gamma^*(s)$} \label{fig3}
\end{center}
\end{figure}

In this situation, we can repeat the arguments of Assertions \ref{ass:1}, \ref{ass:2} and \ref{ass:3}, and prove that there exists a cylinder $\Gamma^*\times \R$, where $\Gamma^*$ is a translation of $\Gamma_0$, such that $\Gamma^*\times \R$ is tangent to $\parc \Omega_a$ at $q_*$ on its concave side, and for which $u$ can be extended at $q^*$ to an exterior tubular neighborhood of $\Gamma^*$ around $q^*$, so that $u\to -\8$ as $q$ approaches $\Gamma^*$ from its concave side. This shows, in particular, that $u(q)$ is well defined for any $q\in \Omega_a$ sufficiently close to $q^*$, what contradicts that $q^*$ is an accumulation point of $\{q_1^n,q_ 2^n\}_n$.

This contradiction proves that the total number of disks $D_i$ appearing in \eqref{decob} for all $b<a$ is a finite number $m$. Recall that we already showed that all the $D_i\cap \Omega_a$ are disjoint, and that every $D_i$, with the possible exception of $D_0$, intersects the exterior of $\Omega_a$.

Then, again by the extension process in \eqref{extens}, the graph $u$ can be extended locally around any point in $\parc \Omega_a$ that belongs to some of the curves $\parc D_i$ to an exterior neighborhood of the point, so that $u\to -\8$ as it approaches $\parc D_i$. 

Consider next a point $q_1\in \parc \Omega_a$ that does not lie in any of the compact disks $D_i$ for any $b<a$. Then, either $u$ extends smoothly across $q_1$ or, by Assertion \ref{ass:1} there is some translation $\Gamma'$ of $\Gamma_0$ such that $\Gamma'$ is tangent to $\parc \Omega_a$ at $q_1$ on its concave side. Moreover, $u\to -\8$ as we approach $q_1$, by Assertion \ref{ass:2}. Call $D'$ to the compact domain bounded by this curve $\Gamma'$.

We claim that two domains $D_1'$ and $D_2'$ constructed in this form cannot intersect inside $\Omega_a$. This follows again directly by the fact that we can extend the function $u$ along the exterior neighborhoods of $\parc D_i'$, $i=1,2$, with $u\to -\8$ along each $\parc D_i'$. Since $\parc \Omega_a$ is compact, there is then a finite number $k\geq 0$ of domains $D'$ obtained in this form.

Finally, all of this proves that $u$ can be extended to $$\tilde{\Omega}_a := \Omega_a \setminus (D_1\cup \dots \cup D_m\cup D_1'\cup \dots \cup D_k')$$ and properties i) and ii) above hold for this extension. In other words, $a\in \mathcal{I}$. Therefore, finally, $\mathcal{I}=[0,1]$.

All of this shows that $S_1$ is the graph of some function $u(x,y)$ defined on the domain
$$
\tilde{\Omega}_1=\Omega_1\backslash(D_1\cup\ldots\cup D_m),\qquad m\in\n.
$$
Moreover, $S_1$ is bounded from above, since $u(q)$ converges to $-\8$ as $q$ approaches $\partial D_i$, $i=1,\ldots,m$.

In this way, we argue as in Case 1; we consider a translation of $2{\cal W}$ whose projection agrees with the compact region bounded by $2\Gamma_0$ and lies above $S_1$. Then, translating $2 \cW$ vertically downwards, we find a first contact point $p_0\in S_1\cap 2{\cal W}$. The point $p_0$ cannot lie in $\partial S_1$, since in that case $\pi (p_0)$ would be a point of $2\Gamma_0$, and the points of $2{\cal W}$ whose projections lie in $2\Gamma_0$ have horizontal unit normal. Thus, $p_0$ is an interior point, and this contradicts the maximum principle.

This final contradiction shows that the complete multigraph $\Sigma$ with CAMC $-1$ cannot exist. This finally completes the proof of Theorem \ref{multigrafo}.
\end{proof}

As a direct consequence of Theorem \ref{multigrafo}, we obtain the desired extension to the anisotropic setting of the Hoffman-Osserman-Schoen theorem (Theorem A in the introduction):
\begin{corollary}\label{corolario1}
Let $\Sigma$ be a complete CAMC surface whose Gauss map image is contained in a closed hemisphere of $\S^2$. Then $\Sigma$ is a plane or a CAMC cylinder. 
\end{corollary}
\begin{proof}
By choosing suitable Euclidean coordinates $(x,y,z)$, we can assume that the third coordinate $N_3$ of the Gauss map $N$ of $\Sigma$ satisfies $N_3\leq 0$. In these conditions, it is well known that either $N_3<0$ everywhere, or else $N_3$ vanishes identically on $\Sigma$; see \cite{KP}. In the first case, we deduce from Theorem \ref{multigrafo} that $\Sigma$ is a plane. In the second case, $\Sigma$ is trivially a CAMC cylinder.
\end{proof}

As an interesting consequence of Theorem \ref{multigrafo}, we can deduce the anisotropic extension of the Klotz-Osserman theorem (Theorem B in the introduction).

\begin{theorem}\label{kno}
Let $\Sigma$ be a complete surface with non-zero CAMC, and whose Gaussian curvature does not change sign. Then $\Sigma$ is a CAMC cylinder or the Wulff shape, up to homotheties. 
\end{theorem}
\begin{proof}
Assume that the Gaussian curvature $K$ of the immersed surface $\Sigma$ is non-negative. Then, by Sacksteder classical theorem \cite{Sa}, there are three options:
\begin{enumerate}
\item
$K$ vanishes identically on $\Sigma$.
 \item
$\Sigma$ is an embedded surface diffeomorphic to $\S^2$.
 \item
$\Sigma$ is complete, non-compact, embedded, and the boundary of some convex set of $\R^3$.
\end{enumerate}

In the first case, $\Sigma$ is a CAMC cylinder. In the second case, $\Sigma$ is the Wulff shape, by \cite{KoisoPalmer} (see also \cite{CCC}). In the third case, the Gauss map image of $\Sigma$ lies in a closed hemisphere of $\S^2$ (by convexity). So, by Corollary \ref{corolario1}, $\Sigma$ is again a CAMC cylinder.

Assume now that $K\leq0$ at every point. Up to ambient homothety, we can assume that the anisotropic mean curvature $H$ of $\Sigma$ is $H=-1$. Let $p\in\Sigma$, and $\{e_1,e_2\}$ be an orthonormal basis of $T_p\Sigma$ given by principal directions, i.e. the (Euclidean) Weingarten endomorphism of $\Sigma$ at $p$ is written as
$$
S_p(e_i)=k_i\,e_i,\qquad i=1,2,
$$
where $k_1,k_2$ are the (Euclidean) principal curvatures of $\Sigma$ at $p$. The anisotropic mean curvature of $\Sigma$ is given by the trace of \eqref{1}, and so we have
$$
-1=H(p)=a_{11}k_1+a_{22}k_2,
$$
where $a_{ii}=\langle (S_{\nu(p)})^{-1}(e_i),e_i\rangle>0$, $i=1,2$. Since $k_1k_2\leq0$, we may assume that $k_1\leq0\leq k_2$. Then, if we let $m$ denote the minimum value of the two (positive) principal curvatures of the Wulff shape $\cW$, we have
$$
-1=H(p)\geq a_{11} \,k_1 \geq \frac{1}{m} \, k_1.
$$
In this way, at every $p\in\Sigma$, the principal curvatures $k_1,k_2$ satisfy $k_1\leq -m<0\leq  k_2$. Once here, since $\Sigma$ is complete and the supremum of $k_1$ is negative, we conclude from \cite[The principal curvature theorem]{SX} that $\Sigma$ is a cylinder. This completes the proof.
\end{proof}

\section{Height estimates and properly embedded CAMC surfaces}\label{sec:4}

In this section we will derive some further consequences of Theorem \ref{multigrafo}, applied to the study of CAMC surfaces properly embedded in $\R^3$. We start with the following height estimate for compact graphs of CAMC and planar boundary.

\begin{lemma}\label{compactas}
Given $H_0\neq0$, there exists a constant $C(H_0)>0$ such that the following assertion holds:

Let $P\subset \R^3$ be any plane, and $\Sigma$ be any immersed compact surface in $\R^3$ with $\partial \Sigma \subset P$, so that $\Sigma$ has CAMC $H_0$, and is a multigraph over $P$ (i.e. no tangent plane of $\Sigma$ is orthogonal to $P$). Then, the distance of any point of $\Sigma$ to $P$ is at most $C(H_0)$.
\end{lemma}
\begin{proof}
Arguing by contradiction, assume that there exists a sequence of oriented mutigraphs $\Sigma_n$ with CAMC $H_0$ over planes $P_n\subset \R^3$, with $\partial\Sigma_n\subset P_n$, and points $q_n\in\Sigma_n$ whose distance to $P_n$ is greater than $n$. 

Let $v_n\in\s^2$ be the normal vector to $P_n$ such that $\esiz N_n,v_n\esde >0$, where $N_n$ is the unit normal of $\Sigma_n$. In particular, the unit normal of $\Sigma_n$ at any $p_n\in \Sigma_n$ that lies at a maximum distance from $P_n$ is equal to $v_n$. Up to suitable translations, we can assume that $p_n$ is the origin of $\R^3$, for every $n$. Also, up to a subsequence, we will assume that $\{v_n\}$ converges to some $v_0\in \S^2$. 

By Theorem \ref{curvatura} for the choice $d=1$, we obtain the existence of a constant $C>0$ such that the norm of the second fundamental form of $\Sigma_n$ is bounded by $C$ for any $n$ and any point of $\Sigma_n$ whose distance to $\parc \Sigma_n$ is greater than $1$. In this way, since the distance of $p_n$ to $\partial \Sigma_n$ diverges to $\8$, it follows from Theorem \ref{compacidad} that a subsequence of the $\{\Sigma_n\}$ converges in the ${\cal C}^2$ topology to a complete, possibly not connected, surface $\Sigma_0$ with CAMC $H_0$, that passes through the origin, with unit normal equal to $v_0$ at that point.

Since $\esiz N_n,v_n\esde>0$, we deduce that $\esiz N_0,v_0\esde \geq 0$, where $N_0$ is the unit normal of $\Sigma_0$. It follows then from Corollary \ref{corolario1} that $\Sigma_0$ is a cylinder, and so $\esiz N_0,v_0\esde$ vanishes identically. This contradicts that $N_0=v_0$ at the origin.
\end{proof}

To obtain a general height estimate for non-compact graphs with planar boundary, we will next adapt to the CAMC case a result by Meeks (cf. \cite[Lemma 2.4]{Me}) for the isotropic (CMC) case. We will only sketch the proof, following a slightly simplified version of Meeks' proof appearing in \cite[Theorem 4]{AEG} or \cite[Theorem 6.2]{Espinar}.

\begin{lemma}\label{meeks}
Let $P\subset\r^3$ be a plane and $\Omega\subset P$ a closed (not necessarily bounded) domain. Let $\Sigma\subset \r^3$ be a normal graph over $\Omega$ of some function $u$, so that $\Sigma$ has CAMC $H_0\neq0$, and $u=0$ on $\partial\Omega\subset P$. 
Let $d_{{\cal W}}$ be the extrinsic diameter of the Wulff shape ${\cal W}$, denote $d_0=2\sqrt{3}d_{{\cal W}}$, and let $P_t$ denote the two parallel planes to $P$ at a distance $t>0$. 

Then, for any $t>d_0/|H_0|$, the extrinsic diameter of each connected component of $\Sigma\cap P_t$ is at most $d_0/|H_0|$. 
In particular, all connected components of $\Sigma\cap P_t$ are compact for $t>d_0/|H_0|$.
\end{lemma}
\begin{proof}
Up to ambient homothety, we will assume that the CAMC of $\Sigma$ is $H_0=-2$. Let us also remember that the Wulff shape ${\cal W}$ has CAMC equal to $-2$ for the choice of its exterior unit normal. Also, since ${\cal W}$ is a compact set of diameter $d_{{\cal W}}$, it is contained in some closed ball of $\R^3$ of diameter $\sqrt{3}\,d_{{\cal W}}$.

Take now Euclidean coordinates $(x,y,z)$ in $\R^3$ so that the plane $P$ is the $z=0$ plane. Thus, $\Sigma$ is given as a graph $z=u(x,y)$ over $\Omega\subset \R^2$, with $u=0$ on $\parc \Omega$. 

Assume that there exists $p_1\in\Omega$ so that $u(p_1)>0$. Then, we can translate $\cW$ horizontally so that its lowest point projects vertically to $p_1$, and then move $\cW$ vertically upwards until it is placed above the graph of $u$ over the compact set $\overline{\Omega}\cap B_R$, where $B_R\subset\r^2$ is a disk centered at $p_1$ and of radius $R>d_{{\cal W}}$. Once here, we can translate $\cW$ vertically downwards until it reaches a first contact point with $\Sigma$. Then, the unit normal $N:\Sigma\flecha \S^2$ of $\Sigma$ must point upwards, since otherwise we would have $\cW=\Sigma$, a contradiction. A similar argument proves that if $u(p_2)<0$ for some $p_2\in \Sigma$, then $N$ points downwards. Thus, $u$ cannot change sign, and we can assume without loss of generality that $u\geq 0$ and that $N$ points upwards. 

Choose now any $t\in \R$ with $t>d_0/|H_0|$, which by our initial normalization $H_0=-2$ means $t>\sqrt{3}\,d_{{\cal W}}$. We will suppose that the plane $z=t$ intersects $\Sigma$ transversally (this happens for almost every $t$, by Sard's theorem), and assume by contradiction that there exists a connected component of $\Sigma\cap\{z=t\}$ with diameter greater than $\sqrt{3}\,d_{{\cal W}}$. Take a simple arc $\Gamma\subset\Omega$ so that the maximum Euclidean distance between its endpoints $p_1,p_2\in\r^2$ is greater than $\sqrt{3}\,d_{{\cal W}}$, with $u(p)\geq t$ for all $p\in \Gamma$. We may choose $\Gamma$ so that the Euclidean distance between $p_1$ and $p_2$ is not smaller than the distance between any other two points of  $\Gamma$. Up to an isometric change of coordinates in $\R^3$ given by a rotation around the $z$-axis and a horizontal translation, we may take $p_1=(-x_0,0),p_2=(x_0,0)$, with $x_0>\sqrt{3}\,d_{{\cal W}}/2$.

In this way, the \emph{rectangular surface} (with boundary) $S:=\Gamma\times[0,t]$ is contained in 
$$
{\cal U}=\{(x,y,z)\in\r^3:\ (x,y)\in\Omega, \ 0<z<u(x,y)\}.
$$
In fact, $S$ divides the solid region
$$
{\cal R}=\{(x,y,z)\in\r^3:\ |x|\leq x_0,\ 0\leq z\leq t\}
$$
into two connected components ${\cal R}_1,{\cal R}_2$.

Once here, we can place ${\cal W}$ inside the interior of ${\cal R}_1$, and then move it continuously towards ${\cal R}_2$ without leaving the interior of ${\cal R}$. Let us only consider the piece of $\cW$ that passes through $S$ into the inside of $\mathcal{R}_2$ by means of this continuous translation process. It is clear that this piece of surface cannot touch $\Sigma$, by the maximum principle, since otherwise $\Sigma=\cW$, which is impossible. Hence, $\cW$ completely passes through $S$ by this translation process, until it ends up being contained in ${\cal R}_2\cap{\cal U}$. But once there, we could move $\cW$ vertically upwards until reaching a first contact point with $\Sigma$, what gives again a contradiction with the maximum principle.
\end{proof}

As a direct consequence of the previous two lemmas, we have:

\begin{theorem}\label{meeks2}
For any $H_0\neq0$ there exist a constant $D(H_0)>0$ so that the following assertion holds:

Let $P\subset \R^3$ be any plane, and $\Omega\subset P$ any closed (not necessarily bounded) domain. Let $\Sigma\subset \r^3$ be a normal graph over $\Omega$ of some function $u$, so that $\Sigma$ has CAMC $H_0\neq0$, and $u=0$ on $\partial\Omega\subset P$. Then, the distance of any point of $\Sigma$ to $P$ is at most $D(H_0)$.
\end{theorem}

Let  $s:\r^3\fl\r^3$ denote a symmetry with respect to some plane $P_0$. If the Wulff shape $\cW$ is symmetric with respect to $P_0$, then any surface $\Sigma$ with anisotropic Gauss map  $\nu:\Sigma\fl{\cal W}$ satisfies that
$$
s(\nu(p))=\hat{\nu}(s(p)),\qquad \forall p\in\Sigma,
$$
where $\hat{\nu}:s(\Sigma)\fl{\cal W}$ denotes the anisotropic Gauss map of the symmetric surface $s(\Sigma)$.
In particular, if $\Sigma$ has CAMC, then $s(\Sigma)$ has the same CAMC, and we can apply the Alexandrov reflection technique with respect to planes parallel to $P_0$. In this way we have:

\begin{corollary}\label{Alexandrov}
Assume that the Wulff shape $\cW$ is symmetric with respect to some plane $P_0$. Then there exists some constant $E(H_0)>0$ such that the following assertion holds:

For any compact embedded surface $\Sigma$ with CAMC $H_0$ and boundary contained in a plane $P$ parallel to $P_0$, it holds that the distance of any point of $\Sigma$ to $P$ is at most $E(H_0)$.
\end{corollary}
\begin{proof}
This is a direct consequence of the Alexandrov reflection principle. Indeed, by using such argument it follows that if $h$ denotes the maximum distance of a point of $\Sigma$ to $P$, then the set of points of $\Sigma$ that are at a distance at least $h/2$ from $P$ has to be a graph with respect to the direction of $\R^3$ orthogonal to $P$. So, the result follows from Lemma \ref{compactas}, taking $E(H_0)=2C(H_0)$.
\end{proof}

The next result was proved by Meeks \cite{Me} for CMC surfaces, i.e., for the case where the Wulff shape $\cW$ is the round sphere. The proof of Theorem \ref{separacion} is the same as Meeks', bearing in mind that $\cW$ is contained in a closed ball of diameter $\sqrt{3}d_{\cal W}$, and taking into account the behavior of the anisotropic mean curvature with respect to ambient homotheties explained in Section \ref{sec:2}

\begin{theorem}[Meeks' separation lemma]\label{separacion}
Let $\Sigma$ be a properly embedded surface in $\R^3$ with CAMC $H_0\neq 0$, and diffeomorphic to a closed disk minus an interior point (in particular, $\parc \Sigma$ is non-empty and compact). 

Let $P_1$, $P_2$ be two parallel planes separated by a distance greater than 
\begin{equation}\label{d0}d_0:=\frac{2\sqrt{3}d_W}{|H_0|},
\end{equation}
where $d_{\cal W}$ denotes the diameter of the Wulff shape. Let $P_1^+$, $P_2^+$ denote the two connected components of the complement in $\r^3$ of the open slab between $P_1$ and $P_2$. Then, all connected components of either $\Sigma\cap P_1^+$ or $\Sigma\cap P_2^+$ are compact.
\end{theorem}

A geometric consequence of Corollary \ref{Alexandrov} and Theorem \ref{separacion} is: 
\begin{proposition}\label{lemi}
Assume that the Wulff shape ${\cal W}$ is symmetric with respect to some plane $P_0$, and let $\Sigma$ be a properly embedded surface in $\R^3$ with CAMC $H_0\neq 0$ that is diffeomorphic to a closed disk minus an interior point. Then, $\Sigma$ lies entirely in a half-space of $\R^3$ whose boundary is a plane $P$ parallel to $P_0$.
\end{proposition}
\begin{proof}
Let $P_1$, $P_2$ be two planes parallel to $P_0$, separated by a distance greater than the constant $d_0$ in \eqref{d0}, and so that $\parc \Sigma$ lies in the open slab determined by them. Then, Theorem \ref{separacion} shows that, up to a relabelling of the planes $P_i$, all connected components of $\Sigma\cap P_2^+$ are compact (there might be an infinite number of such components). Once there, Corollary \ref{Alexandrov} shows that any such connected component lies at a distance at most $E=E(H_0)$ from $P_2$. This proves Lemma \ref{lemi}, taking $P$ parallel to $P_2$ and contained in $P_2^+$, at a distance $E(H_0)$ from $P_2$. 
\end{proof}

Proposition \ref{lemi} has a stronger form in the case that $\parc\Sigma =\emptyset$. Recall that a surface $\Sigma$ is said to have \emph{finite topology} if it is diffeomorphic to a compact surface (without boundary) $\overline{\Sigma}$ with a finite number of points removed, $e_1,\ldots,e_m\in\overline{\Sigma}$. The points $e_i$ will be called the \emph{ends} of the surface $\Sigma$. In this way, we have:
\begin{theorem}\label{banda}
Assume that the Wulff shape ${\cal W}$ is symmetric with respect to some plane $P_0$. Then, there exists a constant $G(H_0)>0$ such that if $\Sigma$ is a properly embedded surface in $\R^3$ with CAMC $H_0\neq 0$, finite topology and only one end, then $\Sigma$ lies in an open slab of $\R^3$ of width at most $G(H_0)$, and whose boundary is the union of two planes parallel to $P_0$.
\end{theorem}
\begin{proof}
Let $P$ be a plane parallel to $P_0$ that intersects $\Sigma$, and $(x,y,z)$ be Euclidean coordinates in $\R^3$ so that $P$ corresponds to the plane $z=0$.

Take $R>E(H_0)$, where $E(H_0)$ is the constant given by Corollary \ref{Alexandrov}. Then, if we choose $P_1=\{z=R\}$ and $P_2=\{z=R+d_0\}$ (with $d_0$ given by \eqref{d0}), and observe that there exist points of $\Sigma$ at a distance $R>E(H_0)$ from $P_1$, we deduce from Theorem \ref{separacion} and Corollary \ref{Alexandrov} that all connected components of $\Sigma\cap\{z\geq R+d_0\}$ are compact. Thus, $\Sigma\subset\{z<2R+d_0\}$.

Analogously, we can prove $\Sigma\subset\{z>-(2R+d_0)\}$, what completes the proof.
\end{proof}

Theorem \ref{banda} also shows that if the Wulff shape $\cW$ has \emph{two} linearly independent planes of symmetry, then any properly embedded surface $\Sigma$ in the conditions of the theorem must be contained in a solid cylinder of $\R^3$. In the case where $\cW$ has \emph{three} linearly independent planes of symmetry we obtain a specially interesting consequence, which generalizes Meeks' Theorem C in the introduction:

\begin{theorem}\label{final}
Assume that the Wulff shape ${\cal W}$ is symmetric with respect to three planes $P_1,P_2,P_3\subset \R^3$ with linearly independent normal vectors. Then $\cW$ is, up to homothety, the only properly embedded surface in $\R^3$ with non-zero CAMC, finite topology and at most one end. \end{theorem}
\begin{proof}
Let $\Sigma$ be a surface in the conditions of the theorem. Then, using Theorem \ref{banda} for each plane $P_i$, $i=1,2,3$, we see that $\Sigma$ lies in a bounded set of $\R^3$. Since $\Sigma$ is proper, it must then be compact. Now, the Alexandrov-type theorem for CAMC surfaces (see \cite{CCC}) proves that $\Sigma$ is, up to homothety, the Wulff shape $\cW$.
\end{proof}

\def\refname{References}

\vskip 0.2cm

\noindent José A. Gálvez

\noindent Departamento de Geometría y Topología,\\ Universidad de Granada (Spain).

\noindent  e-mail: {\tt jagalvez@ugr.es}

\vskip 0.2cm

\noindent Pablo Mira

\noindent Departamento de Matemática Aplicada y Estadística,\\ Universidad Politécnica de Cartagena (Spain).

\noindent  e-mail: {\tt pablo.mira@upct.es}

\vskip 0.4cm

\noindent Research partially supported by MINECO/FEDER Grant no. MTM2016-80313-P 

\noindent Marcos P. Tassi

\noindent Departamento de Matemática,\\ Universidade Federal de Sao Carlos (Brazil).

\noindent  e-mail: {\tt mtassi@dm.ufscar.br}


\begin{thebibliography}{9}

\bibitem{A} A.D. Alexandrov, Uniqueness theorems for surfaces
in the large, I, {\it Vestnik Leningrad Univ.} {\bf 11} (1956),
5--17. (English translation: {\it Amer. Math. Soc. Transl.}  {\bf 21} (1962), 341--354).

\bibitem{AEG} J.A. Aledo, J.M. Espinar, J.A. Gálvez, The Codazzi equation for surfaces. {\it Adv. Math.} {\bf 224} (2010), 2511--2530.

\bibitem{BC}
J.L. Barbosa; M.P. do Carmo, Stability of hypersurfaces with constant mean curvature, {\it Math. Z.} {\bf 185} (1984), 339--353.

\bibitem{BGM} A. Bueno, J.A. Gálvez, P. Mira, The global geometry of surfaces with prescribed mean curvature in $\r^3$, {\it Trans. Amer. Math. Soc.}, to appear. 

\bibitem{Clarenz1}
U. Clarenz, Enclosure theorems for extremals of elliptic parametric functionals, {\it Calc. Var. Partial Differential
Equations} {\bf 15} (2002) 313--324.

\bibitem{Clarenz2}
U. Clarenz, H. von der Mosel, On surfaces of prescribed $F$-mean curvature, {\it J. Diff. Geom.} {\bf 213} (2004), 15--36

\bibitem{DHM} B. Daniel, L. Hauswirth, P. Mira, \emph{Constant mean curvature surfaces in homogeneous manifolds}, Korea Institute for Advanced Study, Seoul, Korea, 2009.

\bibitem{Espinar} J.M. Espinar, J.A. G\'alvez, H. Rosenberg, Complete surfaces with positive extrinsic curvature in product spaces, {\it Comment. Math. Helv.} {\bf 84} (2009), 351--386.

\bibitem{ER} J.M. Espinar, H. Rosenberg, Complete constant mean curvature surfaces and Bernstein type theorems in $M^2\times \R$. {\it J. Diff. Geom.} {\bf 82} (2009), 611--628.

\bibitem{FM} I. Fern\'andez, P. Mira, Constant mean curvature surfaces in $3$-dimensional Thurston geometries. In \emph{Proceedings of the International Congress of Mathematicians}, Volume II (Invited Conferences), pages 830--861. Hindustan Book Agency, New Delhi, 2010.  (arXiv.org/abs/1004.4752)

\bibitem{Finn} R. Finn, On equations of minimal surface type, {\it Ann. of Math.} {\bf 60} (1954), 397--416.

\bibitem{GM3} J.A. Gálvez, P. Mira, Uniqueness of immersed spheres in three-manifolds, {\it J. Diff. Geom.}, to appear. 

\bibitem{GM} J. Ge, H. Ma, Anisotropic isoparametric hypersurfaces in euclidean spaces, {\it Ann.
Glob. Anal. Geom.} {\bf 41} (2012), 347--355.

\bibitem{Hauswirth}
L. Hauswirth, H. Rosenberg, J. Spruck, On complete mean curvature $\frac{1}{2}$ surfaces in $\mathbb{H}^2 \times \mathbb{R}$, {\it Comm. Anal. Geom.} {\bf 5} (2008) 989--1005.

\bibitem{CCC} Y. He, H. Li, H. Ma, J. Ge, 
Compact embedded hypersurfaces with constant higher order anisotropic mean curvatures, {\it 
Indiana Univ. Math. J.} {\bf 58} (2009), 853--868.

\bibitem{He2} Y. He, H. Li, Anisotropic version of a theorem of H. Hopf, {\it Ann. Global Anal. Geom.} \textbf{35} (2009), 243--247.

\bibitem{He5}
Y. He, H. Li, A new variational characterization of the Wulff shape, {\it Diff. Geom. Appl.} \textbf{26} (2008), 377--390.

\bibitem{He4} Y. He, H. Li, Integral formula of Minkowski type and new characterization of the Wulff shape, {\it Acta Mathematica Sinica}, English Series \textbf{24} (2008), 697--704.

\bibitem{Hoffman2}
D. Hoffman, R. Osserman, R. Schoen, On the Gauss map of complete surfaces of constant mean curvature in $\mathbb{R}^3$ and $\mathbb{R}^4$, {\it Comment. Math. Helv.} \textbf{57} (1982), 519--531.


\bibitem{Ho0} H. Hopf, Uber Flachen mit einer Relation zwischen den Hauptkrummungen, {\it Math. Nachr.} {\bf 4} (1951), 232--249.

\bibitem{Ho} H. Hopf. {\it Differential Geometry in the Large, volume 1000 of
Lecture Notes in Math.} Springer-Verlag, 1989.

\bibitem{J}  H. Jenkins, On two-dimensional variational problems in parametric form. {\it Arch. Rational. Mech. Anal.} {\bf 8} (1961), 181--206. 

\bibitem{JS1} H. Jenkins, J. Serrin, Variational problems of minimal surface type. I. {\it Arch. Rational Mech. Anal.} {\bf 12} (1963) 185--212.

\bibitem{JS2} H. Jenkins, J. Serrin,  Variational problems of minimal surface type. III. The Dirichlet problem with infinite data, {\it Arch. Rational Mech. Anal.} {\bf 29} (1968), 304--322.


\bibitem{Klotz}
T. Klotz, R. Osserman, Complete surfaces in $\mathbb{E}^3$ with constant mean curvature, {\it Comment. Math. Helv.}, {\bf 41} (1966), 313--318.

\bibitem{KP} M. Koiso, B. Palmer, 
Geometry and stability of surfaces with constant anisotropic mean curvature. {\it Indiana Univ. Math. J.} {\bf 54} (2005), 1817--1852.

\bibitem{KP3} M. Koiso, B. Palmer, Anisotropic capillary surfaces with wetting energy, {\it Calc. Var. Partial Diff. Equations} {\bf 29} (2007), 295--345.

\bibitem{Koiso3}
M. Koiso, B. Palmer, Rolling construction for anisotropic Delaunay surfaces, {\it Pacific J. Math.} \textbf{234} (2008), 345--378.

\bibitem{KoisoPalmer} M. Koiso, B. Palmer, 
Anisotropic umbilic points and Hopf's theorem for surfaces with constant anisotropic mean curvature. {\it Indiana Univ. Math. J.} {\bf 59} (2010), 79--90.

\bibitem{Kuhns} C. Kuhns, B. Palmer, Helicoidal surfaces with constant anisotropic mean curvature, {\it J. Math. Phys.} \textbf{52} (2011), 073506, 14 pp.

\bibitem{Lira} J.H.S. de Lira, M. Melo, Hypersurfaces with constant anisotropic mean curvature in Riemannian manifolds, {\it Calc. Var.} \textbf{50} (2014), 335--364.


\bibitem{Manzano} J.M. Manzano, M. Rodríguez, On complete constant mean curvature vertical multigraphs in $\mathbb{E}(\kappa,\tau)$, {\it J. Geom. Anal.} {\bf 25} (2015), 336--346.

\bibitem{Me} W.H. Meeks, The topology and geometry of embedded surfaces of constant mean curvature, {\it J. Diff. Geom.} {\bf 27} (1988), 539--552.

\bibitem{Palmer}
B. Palmer,  Stabilty of the Wulff shape, {\it Proc. Amer. Math. Soc.} \textbf{126} (1998), 3661--3667.

\bibitem{RST} H. Rosenberg, R. Souam, E. Toubiana, General curvature estimates for stable H-surfaces in 3-manifolds and applications. {\it J. Diff. Geom.} {\bf 84} (2010), 623--648.

\bibitem{Sa} R. Sackstader, On hypersurfaces with non-negative sectional curvatures, {\it Amer. J. Math.} {\bf 82} (1960), 609--630.

\bibitem{SX} B. Smyth, F. Xavier, Efimov's theorem in dimension greater than two. {\it Invent. Math.} {\bf 90} (1987), 443--450.

\bibitem{Taylor}
J. Taylor, Crystalline variational problems, {\it Bull. Amer. Math. Soc.} \textbf{84} (1987), 568--588.





\end{thebibliography}
\end{document}